\documentclass[11pt]{article}
\usepackage{amsmath, amssymb, amscd, amsthm, amsfonts}
\usepackage{graphicx}
\usepackage{hyperref}
\usepackage{graphicx}%
\usepackage{multirow}%
\usepackage{amsmath,amssymb,amsfonts}%
\usepackage{amsthm}%
\usepackage[title]{appendix}%
\usepackage{xcolor}%
\usepackage{textcomp}%
\usepackage{manyfoot}%
\usepackage{algorithm}%
\usepackage{algorithmicx}%
\usepackage{algpseudocode}%
\usepackage{listings}%
\usepackage[caption=false]{subfig}
\newtheorem{remark}{Remark}%
\usepackage{hyperref} 

\newcommand{\email}[1]{\href{mailto:#1}{\nolinkurl{#1}}}
\makeatletter
\expandafter\newcommand\csname e-mail\endcsname[1]{\href{mailto:#1}{\nolinkurl{#1}}}
\makeatother

\title{Soliton Solutions to the Curvature Flow on the 2-dimensional De Sitter Space and Applications}

\author{F\'abio Nunes da Silva\footnote{ Universidade Federal do Oeste da Bahia, Brazil,  fabionuness@ufob.edu.br} 
\qquad Edwin Salinas Reyes\footnote{ Universidade Federal do Oeste da Bahia, Brazil, edwin.reyes@ufob.edu.br}
  }

\date{}

\newtheorem{theorem}{Theorem}
\newtheorem{lemma}[theorem]{Lemma}

\newtheorem{proposition}[theorem]{Proposition}

\begin{document}

\maketitle

\begin{abstract}
We show that the spacelike solutions to the curvature flow for curves on the De Sitter space are in correspondence with the solutions to the inverse curvature flow on the 2-dimensional hyperbolic space,  that on the De Sitter space, the timelike solutions to the curvature flow are in correspondence with the timelike solutions to the inverse curvature flow, and that the solutions curve shortening flow on the 2-dimensional hyperbolic space are in correspondence with the solutions to the spacelike solutions to the inverse curvature flow on the De Sitter space. We prove that, for spacelike curves on the De Sitter space, the curvature flow is a gradient-type flow for the arc-length functional. We observe that a spacelike or timelike curve on the De Sitter space is a soliton solution to the curvature flow (resp. inverse curvature flow) if and only if its curvature (resp. inverse of its curvature) can be written as the inner product between its tangent vector field and a fixed vector $v$ of the 3-dimensional Minkowski space. We prove that for each vector $v$, there exists a 2-parameter family of timelike (spacelike) soliton solutions to the curvature flow and to the inverse curvature flow on the De Sitter space. We show that there exists no non-trivial complete timelike soliton. There exist non-trivial complete spacelike solutions. As a consequence of curvature flow, we obtain the behavior of the soliton solutions to the inverse curvature flow on the De Sitter space and hyperbolic space.
\end{abstract}
\vspace{0.5cm}
\section{Introduction}	  
    A 1-parameter family of  curves $\displaystyle \hat{X}^{t}: I \rightarrow M^{2}$, $t  \in [0, T)$  on 2-dimensional semi-Riemannian manifold $M^2$ is a \it solution to the curvature flow  \rm (CF) (respectively \it inverse curvature flow \rm (ICF)), with initial condition $X(u)$, $u\in I$, if  
\begin{equation}\label{cficf}
\left\{\begin{array}{ll}
\displaystyle{<\frac{\partial }{\partial t}\hat{X}^{t}(\cdot), \hat{N}^t(\cdot)>}=\hat{k}^{t}(\cdot)\\\
\hat{X}^{0}(\cdot)=X(\cdot),
\end{array}\right. 
\qquad 
\left( respectively  \left\{\begin{array}{ll}
\displaystyle{<\frac{\partial }{\partial t}\hat{X}^{t}(\cdot), \hat{N}^t(\cdot)>}=-\frac{1}{\hat{k}^{t}(\cdot)}\\
\hat{X}^{0}(\cdot)=X(\cdot),
\end{array}\right. \right)
\end{equation}
where $\hat{k}^{t}(\cdot)=\hat{k}(\cdot,t)$ is the curvature of $\hat{X}^{t}(\cdot)=\hat{X}(\cdot,t)$ and $\hat{N}^{t}(\cdot)=\hat{N}(\cdot,t)$ is the normal vector field of $\hat{X}^{t}(\cdot)$ for each $t\in J$. 
When $X(u)$ is a geodesic i.e. $k\equiv 0$, then the family $\hat{X}^{t}(u)=X(u)$, for all $t$,  is a {\it trivial solution} to the CF. Solutions that evolve under isometries and/or homotheties are called \it self-similar solutions, \rm and the case that evolves under isometries is called a soliton solution. \rm When $M^2$ is a 2-dimensional Riemannian manifold, the curvature flow is called \it curve shortening flow \rm and is a gradient-type flow for the length functional. But, in general, curvature flow is not gradient-type; see Halldorsson \cite{Halldorsson1} and Da Silva and Tenenblat \cite{Silva1}, who studied self-similar solutions to the curvature flow on the Minkowski plane and the 2-dimensional lightlike cone, respectively. 

In the 2-dimensional Euclidean space, the curve shortening flow has been studied by several authors. Abresch and Langer \cite{Abresch} and Halldorsson \cite{Halldorsson} gave a complete description of the self-similar solutions for curves on the plane. The self-similar solutions of curve shortening flow were very important for the study of your behaviour: Gage and Hamilton \cite{Gage3} showed that, under the flow, convex embedded closed curves in the plane preserve their convexity, become circular, and then shrink to a point; Grayson \cite{Grayson1, Grayson} proved that closed embedded curves evolve to circular shapes and subsequently collapse to a point in finite time and Angenent \cite{Angenent}, under more general conditions, proved that the flow shrink to a point in an asymptotically self-similar manner. Some authors have also studied the flow when $M^2$ differs from the plane, as seen in \cite{Gage2, Ma, Zhou}. For the self-similar solutions, Dos Reis and Tenenblat \cite{DosReis} described all the soliton solutions on the 2-dimensional sphere, Da Silva and Tenenblat \cite{Silva} also described all the soliton solutions on the 2-dimensional hyperbolic space, and Woolgar and Xie \cite{Eric} studied the self-similar solutions on the 2-dimensional hyperbolic plane.

The inverse curvature flow has also been extensively studied, both in the context of self-similar solutions and in the general setting; see \cite{Andrews, Chang, Drugan, Kroner, KWONG, Urbas1, TIAN} for a 2-dimensional Riemannian manifold. Da Silva and Tenenblat \cite{Silva1} studied self-similar solutions of curvature flow and established a relationship between curvature flow and inverse curvature flow on the 2-dimensional lightlike cone. 

 In this paper, we study the curvature flow and the inverse curvature flow on the Sitter Space. Initially, in the Theorem \ref{12}, we prove the follow relationships between CF and ICF: the spacelike solutions to the curvature flow (CF) for curves on the 2-dimensional De Sitter Space are in correspondence with the solutions to the inverse curvature curve (ICF) on the 2-dimensional Hyperbolic Space; the timelike solutions to the curvature flow (CF) for curves on the 2-dimensional De Sitter Space are in correspondence with the timelike solutions to the inverse curvature flow (ICF) on the 2-dimensional De Sitter Space and the solutions curve shortening flow on the 2-dimensional Hyperbolic Space are in correspondence with the solutions to the spacelike solutions to the inverse curvature flow (ICF) on the 2-dimensional De Sitter Space. We verify that, for spacelike curves, the flow is a gradient-type flow for the arc-length functional, in essence, the curve shortening flow. In the next, we present in the Theorem \ref{c3t2} the results: a spacelike or timelike curve on the 2-dimensional De Sitter Space is a soliton solution to the CF (respectively ICF) if and only if its curvature (respectively inverse of its curvature) can be written as the inner product between its tangent vector field and a fixed vector $v$ of the 3-dimensional Minkowski space. 
 
 We prove in Theorem \ref{c9} that for each vector $v$, there exists a 2-parameter family of timelike soliton solutions to the CF on the Sitter Space. There is no complete solution, and at each end of such a curve, the curvature is unbounded or tends to zero. As a consequence of Theorems \ref{12} and \ref{c9}, we obtain Theorem \ref{f1} about timelike soliton solutions of ICF on the Sitter Space and show that there are no complete solutions; at each end of such a curve, the curvature is unbounded or converges to zero.

 For the spacelike curves, we prove in Theorem \ref{c8} that for each vector $v$, there exists a 2-parameter family of spacelike soliton solutions to the CF in the 2-dimensional De Sitter Space. Moreover, at each end of such a curve, the curvature is either unbounded or tends to $\pm1$ or $0$. There exist non-trivial complete spacelike solutions; in this case, the curvature is bounded. As a consequence of Theorems \ref{12} and \ref{c8}, we show Theorem \ref{f2} about the soliton solutions of ICF on the 2-dimensional hyperbolic space, and that at each end of such a curve, the curvature is either unbounded or it tends to $\pm1$ or to $0$.

 We present some graphs of soliton solutions of CF and ICF on the 2-dimensional De Sitter Space and soliton solutions of ICF on the 2-dimensional hyperbolic space.

    We consider the 3-dimensional Minkowski space as $\mathbb{R}_1^{3}=(\mathbb{R}^{3}, \langle, \rangle )$, where $\mathbb{R}^3$ is the 3-dimensional vector space and $\langle, \rangle$ is the Minkowski metric defined by $\langle u,v \rangle=-u_1v_1+u_2v_2+u_3v_3.$ We define the 2-dimensional \it  De Sitter space \rm as the timelike surface $\mathbb{S}_{1}^{2}:= \{p=(p_1,p_2,p_3) \in \mathbb{R}_{1}^{3}:\langle p,p \rangle=1\}$ and the 2-dimensional \it hyperbolic space \rm as the spacelike surface $\mathbb{H}^{2}:= \{p=(p_1,p_2,p_3) \in \mathbb{R}_{1}^{3}:\langle p,p\rangle=-1, p_1>0\}.$ 
	
	Let $\displaystyle X:I\subset \mathbb{R} \rightarrow \mathbb{S}_{1}^{2}\subset \mathbb{R}_1^{3}$ be a spacelike or timelike curve parametrized by arc length $s$. The curve $X$ is characterized by a trihedrom $\{X(s),T(s),N(s)\}$, $s \in I$, where $T(s)=X^{\prime}(s)$ is the unit tangent vector field and $\langle T(s),T(s)\rangle=\epsilon$ such that $\epsilon=1$ when $X(s)$ is a spacelike curve and $\epsilon=-1$ when $X(s)$ is a timelike curve and $N(s)=X(s)\times T(s)$ is the normal vector field orthogonal to $T(s)$ and $X(s)$ such that  $\langle N(s), N(s)\rangle=-\epsilon$, for all $s\in I$ . The {\it geodesic curvature} $k(s)$ of $X(s)$, at $s \in I$, is defined by
	\begin{equation}\label{defk}
		k(s)=\langle T^{\prime}(s), N(s)\rangle.
	\end{equation}
	Moreover, $N^{\prime}(s)=-\epsilon k(s)T(s)$ and $\epsilon T^{\prime}(s)=-X(s)- k(s)N(s)$. Moreover, $N(s)\times T(s)=\epsilon X(s).$
  
  Analogously, taking $\displaystyle X:I\subset \mathbb{R} \rightarrow \mathbb{H}^{2}\subset \mathbb{R}_1^{3}$ a regular curve parametrized by arc length $s$. We denote by $T(s)=X^{\prime}(s)$ the tangent vector field, $N(s)=X(s) \times T(s)$ the unit normal vector field, $k(s)=\langle T^{\prime}(s), N(s)\rangle $ the geodesic curvature of $X$,  $N^{\prime}(s)=-k(s)T(s)$ and $T^{\prime}(s)=X(s)+k(s)N(s)$. In this case, $-\langle X(s), X(s) \rangle=\langle T(s), T(s) \rangle=\langle N(s), N(s) \rangle=1$ and $N(s)\times T(s)=X(s).$

Let $X(u)$, $u \in I$, be a regular, spacelike or timelike parameterized curve with parameter $u \in I$ on $\mathbb{S}^2_1$ or $\mathbb{H}^2$. Then the geodesic curvature $k(u)$, of $X(u)$,  at $u \in I$, is given for 
\begin{equation}\label{k}
k(u)=\frac{1}{[w(u)]^3}<X''(u),X(u)\times X'(u)>,
\end{equation} where $w(u)= \Vert X'(u)\Vert$. 

Let $\displaystyle X:I\subset \mathbb{R} \rightarrow \mathbb{S}_1^{2}\subset \mathbb{R}_1^{3}$ be a regular curve parametrized by arc length $s$ with $k(s)\neq 0$ for all $s \in I$. Taking the curve $Y(s)=N(s)$, we obtain $w(s)=|k(s)|$, $T_Y(s)=-\epsilon T(s)$ the tangent vector field, $N_Y(s)=Y(s)\times T_Y(s)=N(s)\times \left(-\epsilon T(s) \right)=-\epsilon \left(\epsilon X(s)\right)=-X(s)$   the normal vector field of curve $Y(s)$ and using the equation \eqref{k} we obtain your the geodesic curvature is given by 

\begin{eqnarray*}
k_Y(s) &=& \frac{1}{|k(s)|^{3}}\langle N^{\prime \prime}(s), N(s) \times N^{\prime}(s)\rangle\\
       &=& \frac{1}{|k(s)|^{3}}\langle -\epsilon k^{\prime}T(s)-k(s)[-X(s)- k(s)N(s)], N(s) \times -\epsilon k(s)T(s)\rangle\\
       &=& \frac{\epsilon k^{2}(s)}{|k(s)|^{3}}\langle X(s),T(s) \times N(s) \rangle \\
       &=& \frac{\epsilon k^{2}(s)}{|k(s)|^{3}}\langle X(s),-\epsilon X(s) \rangle \\
       &=& -\frac{1}{|k(s)|}.
\end{eqnarray*}
In similar manner, taking $\displaystyle X:I\subset \mathbb{R} \rightarrow \mathbb{H}^{2}\subset \mathbb{R}_1^{3}$ a regular curve parametrized by arc length $s$ with $k(s)\neq0$. Then $Y(s)=N(s)$ is a spacelike curve on $\mathbb{S}^{2}_1$, $T_Y(s)=- T(s)$ the tangent vector field, $N_Y(s)=N(s)\times (-T(s))=-X(s)$ the normal vector field of curve $Y(s)$ and your geodesic curvature is given by $k_Y(s)=-\frac{1}{|k(s)|}$. 

Note that, if $X(u)$, $u \in I$ a spacelike or timelike regular parameterized curve of parameter $u \in I$ on $\mathbb{S}^2_1$ or $\mathbb{H}^2$ with $k(s)\neq0$, then, without loss of generality, we can take $k(u)>0$, because $X(u)$ and $X(-u)$, $u \in I$  have the same set and the geodesic curvatures have opposite sign. Thus, the normal vector field of $Y(u)=N(u)$ and the geodesic curvature can be given by 
\begin{equation} \label{N}
	k_Y(u)=-\frac{1}{k(u)}\,\,\,\text{and}\,\,\, N_Y(u)=-X(u).
\end{equation}

Now, we consider $\hat{X}^{t}(u)$ a 1-parameter family of spacelike or timelike curves on $\mathbb{S}^{2}_1$ or $\mathbb{H}^{2}$, thus $\hat{N}^{t}(u)$, the normal vector field for each $t$, is a 1-parameter family of spacelike or timelike curves on $\mathbb{S}^{2}_1$ or $\mathbb{H}^{2}$. Thus, $\langle \hat{X}^{t}(u), \hat{N}^{t}(u) \rangle =0 $ and 

\begin{equation}\label{x}
\langle \frac{\partial}{\partial t}\hat{X}^{t}(u), \hat{N}^{t}(u) \rangle=- \langle \hat{X}^{t}(u), \frac{\partial}{\partial t}\hat{N}^{t}(u)\rangle.
\end{equation}
Given the foregoing, we enunciate the theorem.

\begin{theorem} \label{12}
    Let $\hat{X}^{t}(u)$ be a 1-parameter family of spacelike or timelike curves on $\mathbb{S}^{2}_1$ or $\mathbb{H}^{2}$ and, $\hat{N}^{t}(u)$, the normal vector field for each $t$ and $\hat{k}^{t}(u)\neq 0$ for all $u \in I$ and $t \in J$. Then:
    \begin{itemize}
        \item[i)] A family of spacelike curves  $\hat{X}^{t}(u)$ is a solution to the curvature flow (CF) on the 2-dimensional De Sitter Space if and only if $\hat{N}^{t}(u)$ is a solution to the inverse curvature curve (ICF) on the 2-dimensional hyperbolic space.
        \item[ii)] A family of timelike curves  $\hat{X}^{t}(u)$ is a solution to the curvature flow (CF) on the 2-dimensional De Sitter Space if and only if $\hat{N}^{t}(u)$ is a timelike solution to the inverse curvature curve (ICF) on the 2-dimensional De Sitter Space.
        \item[iii)] A family of curves  $\hat{X}^{t}(u)$ is a solution to the curvature flow (CF) on the 2-dimensional hyperbolic space if and only if $\hat{N}^{t}(u)$ is a spacelike solution to the inverse curvature curve (ICF) on the 2-dimensional De Sitter Space.
    \end{itemize}
\end{theorem}
\begin{proof}
  From equations \eqref{N} we can consider  $k_{N^t}=-\frac{1}{k^t(u)}$ and your normal vector field is given by $-\hat{X}^{t}(u)$. Thus, 
  $$k^{t}(u)=-\frac{1}{\displaystyle- \frac{1}{k^t(u)}}=\displaystyle -\frac{1}{k_{N^t}}.$$
Therefore, it follows from equation \eqref{x} that a 1-parameter family $X^t(u)$ is
a solution to the curvature flow,
$\langle \frac{\partial}{\partial t}\hat{X}^{t}(u), \hat{N}^{t}(u) \rangle=k^{t}(u),$ if and only if $\hat{N}^{t}(u)$ to the inverse curvature curve (ICF), 

  $$\langle \frac{\partial}{\partial t}\hat{N}^{t}(u), -\hat{X}^{t}(u)\rangle =\langle \frac{\partial}{\partial t}\hat{X}^{t}(u), \hat{N}^{t}(u) \rangle=k^{t}(u)=\displaystyle -\frac{1}{k_{N^t}}.$$  
\end{proof}
The definition of curvature flow was motivated by the curve shortening flow for curves on a 2-dimensional manifold $M^2$, where one considers the inner product 	$<\partial/\partial t\;\hat{X}^{t}, \hat{N}^t>=\hat{k}^{t}$ where $ \hat{N}^t$ is the unit vector field normal to the curve. This flow is a gradient-type flow for the length functional. In our case, let $\displaystyle \hat{X}: I\times J \rightarrow \mathbb{S}_{1}^{2}\subset \mathbb{R}_1^{3}$ be a  1-parameter family of spacelike closed curves, $I=[a,b]$ and $X^t(a)=X^t(b)$ for all $ t\in J$. The  length functional is given by $L(t)=\int_a^b{w(u,t)du}$, where $w(u,t)=|\frac{\partial}{\partial u}\bar{X}(u,t)|$. Using the fact that the variables $u$ and $t$ are independent, $\hat{X}^t(\cdot )$ is a solution of the curvature flow on $\mathbb{S}_1^2$ and the equation \eqref{k}, we obtain:

\begin{eqnarray*}
    \frac{d}{dt}L(t)&=&\int_a^b{\frac{1}{w(u,t)}<\frac{\partial}{\partial u}[\hat{k}(u,t) \hat{N}(u,t)],\frac{\partial}{\partial u}\hat{X}(u,t)>du}\\
    &=& \int_a^b{\frac{\partial}{\partial u}<\hat{k}(u,t) \hat{N}(u,t),\frac{1}{w(u,t)}\frac{\partial}{\partial u}\hat{X}(u,t)>du}\\
    &-&\int_a^b{\hat{k}(u,t)< \hat{N}(u,t),\frac{\partial}{\partial u}\left[\frac{1}{w(u,t)}\frac{\partial}{\partial u}\hat{X}(u,t)\right]>du}\\
     &=&-\int_a^b{\frac{1}{w(u,t)}\hat{k}(u,t)< \hat{X}(u,t)\times \frac{\partial}{\partial u}\hat{X}(u,t),\frac{\partial^2}{\partial u^2}\hat{X}(u,t)>du}\\
     &=&-\int_a^b{w^2(u,t)\hat{k}^2(u,t)du}\leq 0,
\end{eqnarray*}
that is, for spacelike curves, the curvature flow is a gradient-type flow for the arc-length functional.         
			
We investigate solutions to the CF and ICF that evolve by isometries of $\mathbb{S}_1^{2}$. We remark that an isometry of $\mathbb{S}_{1}^{2}$ is an element of the Lie group $O_1(3)$, acting on $\mathbb{R}^3_1$, that preserves $\mathbb{S}_1^{2}$. 
Let	$\displaystyle \hat{X}: I\times J \rightarrow \mathbb{S}_1^{2}$ be a solution to the CF (respectively, ICF) on $\mathbb{S}_1^{2}$, with initial condition $\displaystyle X: I \rightarrow \mathbb{S}_1^{2}$.  
The curve $X$ is a soliton solution to the CF (respectively, ICF) if $\hat{X}^{t}(s)=M(t)X(s)$,  where $M(t)$, $t \in J$ is a family of isometries of $\mathbb{S}_1^{2}$, such that $M(0)=Id$ is the identity map. 
In this context, the soliton solutions of the CF and ICF are as follows
\begin{proposition}\label{p11}
 Let $X(s)$ be a spacelike or timelike curve parametrized by arc length $s$ on $\mathbb{S}^{2}_1$ or $\mathbb{H}^{2}$ with $k(s)\neq 0$ for all $s \in I$ and $N(s)$ the normal vector field. Then, $X(s)$ is a soliton solution to the CF on $\mathbb{S}^{2}_1$ or $\mathbb{H}^{2}$ if and only if, $N(s)$ is a soliton solution to the ICF on $\mathbb{S}^{2}_1$ or $\mathbb{H}^{2}$.
\end{proposition} 	
\begin{proof}
Observe that if $\hat{X}^{t}(s)=M(t)X(s)$ then $\hat{N}^{t}(s)=M(t)N(s)$. Thus, the result follows from Theorem \ref{12}. 
\end{proof}
The following result characterizes the soliton solutions to the CF and ICF. Since the isometry group is the same as the isometry group of $\mathbb{H}^2$, the proof is analogous to the characterization theorem for soliton solutions to the Curve Shortening Flow on $\mathbb{H}^{2}$, see \cite{Silva}, Theorem 2.2. We will omit the demonstration. 
	
\begin{theorem} \label{c3t2}
Let $\displaystyle X: I \rightarrow \mathbb{S}_1^{2}$ be a spacelike or timelike curve parametrized by arc length $s\in I$. Then $X$ is a soliton solution to the CF (respectively, ICF, $k(s)\neq0$) on $\mathbb{S}_1^{2}$ if, and only if, there exists a vector $v \in \mathbb{R}^{3}_1\setminus \{0\}$ such that
\begin{equation}\label{eq2}
	\langle T(s),v\rangle=k(s),\,\,\left(\text{respectively,} \,\,\langle T(s),v\rangle=\frac{1}{k(s)},\right)
\end{equation}
where $T$ is the unit tangent vector field and $k$ is the geodesic curvature  of $X$.
	\end{theorem}
As a consequence of the characterization given by Theorem \ref{c3t2}, we can show that obtaining soliton solutions to the CF corresponds to obtaining solutions to systems of ordinary differential equations. We obtain the following theorems, which describe the soliton solutions to the CF on $\mathbb{S}^{2}_{1}$.
\begin{theorem}\label{c9}
	For any $v \in \mathbb{R}^{3}_1\setminus\{0\}$, there is a 2-parameter family of non-trivial timelike soliton solutions to the curvature flow on the 2-dimensional De Sitter space. There exists no non-trivial timelike complete soliton. The curvature function has at most one zero. Moreover, we have three cases: the curvature function has no zeros and is unbounded at both ends; the curvature function has a unique zero and is unbounded at both ends; and the curvature function has no zeros, is unbounded at one end, and converges to zero at the other.
 \end{theorem} 
 
\begin{theorem}\label{c8}
	For any $v \in \mathbb{R}^{3}_1\setminus\{0\}$, there is a 2-parameter family of non-trivial spacelike soliton solutions to the curvature flow on the 2-dimensional De Sitter space. The curvature function at each end is either unbounded or tends to one of the following constants $\{-1,0,1\}$. The curvature function can have no zeros, finitely many zeros, or infinitely many zeros. There exist non-trivial spacelike complete soliton solutions. Moreover, there exist non-trivial spacelike soliton solutions such that the curvature function has no zeros and is unbounded at each end. 
\end{theorem} 

In what follows, motivated by Theorem \ref{c3t2}, we present a series of properties and lemmas to prove Theorems \ref{c9} and \ref{c8}. First, we use the equation \eqref{eq2} to establish a relationship between soliton solutions of the CF on the 2-dimensional De Sitter space and the solutions of a system of ordinary differential equations.

\begin{proposition} \label{c4p1}
	Let $X: I \rightarrow \mathbb{S}^{2}_{1}$ be a spacelike or timelike curve parametrized by arc length $s$. Consider the vectors
	\begin{equation} \label{e_i}
		e_{1}=(-1,0,0),\hspace{0,3 cm} e_{2}=(-1,1,0), \hspace{0,3 cm} \text{e}\hspace{0,3 cm} e_{3}=(0,0,1).
	\end{equation}
	For each $i \in \{1,2,3\}$, define the functions 
	\begin{equation}\label{alphai}	
		\alpha_{i}(s)=\langle X(s),e_{i}\rangle, \qquad \tau_{i}(s)=\langle T(s),e_{i}\rangle, \qquad  \eta_{i}(s)=\langle N(s),e_{i}\rangle,
	\end{equation}	
	where $T$ is the unit tangent vector with $\langle T(s), T(s)\rangle=\epsilon$ and $N$ is the normal vector field to $X$. For fixed $a>0$,
	\begin{equation*}
		k(s)=a\tau_{i}(s)=k_{i}(s)
	\end{equation*}
	is satisfied for all $s \in I$ if, and only if, the functions $\alpha_{i}(s)$, $\tau_{i}(s)$ and $\eta_{i}(s)$ satisfy the system
	\begin{equation}\label{c4e1}
		\left\{\begin{array}{lll}
			\alpha^{\prime}_{i}(s)=\tau_{i}(s),\\
			\tau^{\prime}_{i}(s)=-\epsilon \alpha_{i}(s)-\epsilon a\tau_i(s)\eta_{i}(s),\\
			\eta^{\prime}_{i}(s)=-\epsilon a\tau^{2}_{i}(s),
		\end{array}\right.
	\end{equation}
	such that 
	\begin{equation}\label{e_i2}
		\alpha^2_{i}(s)+\epsilon\tau^{2}_{i}(s)-\epsilon\eta^2_{i}(s)=
		\left\{\begin{array}{lll}
			-1, \hspace{.3 cm} \text{if} \hspace{.3 cm} i=1,\\
			0, \hspace{.3 cm} \text{if} \hspace{.3 cm} i=2,\\
			1, \hspace{.3 cm} \text{if} \hspace{.3 cm} i=3,
		\end{array}\right.
	\end{equation}
	for all $s \in I$. Moreover, $\eta^{\prime}_{i}(s)$ is a decreasing (respectively increasing ) function if $X$ is  a spacelike curve  (respectively timelike curve).
\end{proposition}

\begin{proof}
	The vector fields $X$, $T$ and $N$ satisfy the Frenet-Serret equations 
	\begin{equation}\label{c4e3}
		\left\{\begin{array}{lll}
			X^{\prime}(s)=T(s),\\
			T^{\prime}(s)=-\epsilon X(s)-\epsilon k(s)N(s),\\
			N^{\prime}(s)=-\epsilon k(s)T(s),
		\end{array}\right.
	\end{equation}
	for all $s \in I$. Thus, taking the inner product with $e_i$, we obtain 
	\begin{equation}
		\left\{\begin{array}{lll}\label{c4e2}
			\alpha^{\prime}_{i}(s)=\tau_{i}(s),\\
			\tau^{\prime}_{i}(s)=-\epsilon\alpha_{i}(s)- \epsilon k(s)\eta_{i}(s),\\
			\eta^{\prime}_{i}(s)=-\epsilon k(s)\tau_{i}(s).
		\end{array}\right.
	\end{equation}
	Assume that $k(s)=k_{i}(s)=a\tau_{i}(s)$ for all $s \in I$, then  \eqref{c4e1} is satisfied. Moreover, since\newline $\{X(s),T(s),N(s)\}$ is a basis for $\mathbb{R}_{1}^{3}$ for each $s$, it follows from \eqref{alphai} that  $e_i=\alpha_i(s)X(s)+\epsilon \tau_i(s)T(s)-\epsilon \eta_i(s)N(s)$ i.e. $\langle e_i,e_i\rangle=\alpha^2_{i}(s)+\epsilon\tau^{2}_{i}(s)-\epsilon\eta^2_{i}(s)$ for all $s \in I$ and for all $i \in \left\{1,2,3 \right\}$. 
	
	Conversely, suppose that the functions $\alpha_{i}(s)$, $\tau_{i}(s)$ and $\eta_{i}(s)$ satisfying \eqref{c4e1} and \eqref{e_i2} for each $i \in \{1,2,3\}$. Since \eqref{c4e2} holds, then $\left[a\tau_{i}(s)-k(s)\right]\eta_{i}(s)=0$ and 
	$\left[a\tau_{i}(s)-k(s)\right]\tau_{i}(s)=0$, 
	for all $s \in I$. Assume by contradiction that $a\tau_{i}(s_0)-k(s_0)\neq 0$ for some  $s_0 \in I$. Then $\eta_{i}(s_0)=\tau_{i}(s_0)=0$   and hence  $\langle e_i, e_i\rangle=\alpha_i^2(s_0)$, which is a contradiction with \eqref{e_i2}. Therefore, $k_{i}(s)=a\tau_{i}(s)$ for all $s \in I$.
\end{proof}
Our next proposition shows that a solution of the system \eqref{c4e1}, with initial conditions satisfying \eqref{e_i2}, is related to a soliton solution to the CF. 

\begin{proposition} \label{c4p2}
	Given a solution $(\alpha(s), \tau(s), \eta(s))$ of the system \eqref{c4e1} on some interval $I$  with fixed $a>0$,  and initial conditions $(\alpha(0),\tau(0),\eta(0))$ satisfying \eqref{e_i2}, there exists a smooth  spacelike curve or timelike curve  $X: I \rightarrow \mathbb{S}_{1}^{2}$, parametrized by arc length $s$,   such that its tangent and normal unit vetor fields $T$ and $N$  satisfy
	\begin{equation}\label{c4e5}
		\alpha(s)=\langle X(s),e\rangle,\hspace{.5 cm}\tau(s)=\langle T(s),e\rangle \hspace{.5 cm} \text{\rm e} \hspace{.5 cm} \eta(s)=\langle N(s),e\rangle, 
	\end{equation} 
	where $e=(-1,0,0)$  (respectively $e=(-1,1,0)$ and $e=(0,0,1)$).
\end{proposition}
\begin{proof}
	Define $k(s)=a\tau(s)$. Thus, up to isometries of $\mathbb{S}_{1}^{2}$, there exists a unique timelike  curve or spacelike curve  $\displaystyle X: I \rightarrow \mathbb{S}_{1}^{2}$, whose curvature is $k(s)$ i.e. $X(s)$ and its tangent and normal unit  vector fields $T(s)$ and $N(s)$ satisfy \eqref{c4e3}.  The curve $X(s)$ is uniquely determined by the initial conditions $X(0)$, $T(0)$ and $N(0)$, that can be chosen such that $$\alpha(0)X(0)+\epsilon \tau(0)T(0)-\epsilon \eta(0)N(0)=e,$$ 
	where $e=(-1,0,0)$ (respectively $e=(-1,1,0)$ and $e=(0,0,1)$).    A straightforward computation shows that \eqref{c4e1} and \eqref{c4e3} imply
	$$	\frac{d}{ds}\left[\alpha(s)X(s)+\epsilon \tau(s)T(s)-\epsilon \eta(s)N(s)\right]=0.$$ Therefore, \eqref{c4e5} is satisfied.
\end{proof}

\begin{remark}
\label{ob1}
	From here, suppose that $\displaystyle X: I \rightarrow \mathbb{S}_{1}^{2}$ is a timelike curve or spacelike curve parametrized by arc length $s$, such that $X(s)=(x_1(s),x_2(s), x_3(s))$ is the position vector field. The function $\alpha(s)$ defined by \eqref{c4e5} has the following geometric interpretation.
	\begin{enumerate}
	\item If $e=(-1,0,0)$ (timelike vector), then $\alpha(s)=x_1(s)$  for all $s \in I$. Moreover, $\alpha(s)$ is the Euclidean height function (with sign) for the vector $(1,0,0).$
		\item If $e=(-1,1,0)$ (lightlike vector), then $\alpha(s)=x_1(s)+x_2(s)> 0$  for all $s \in I$. Moreover, $\alpha(s)$ is the Euclidean height function (with sign) for the vector $(1,1,0).$
		\item If $e=(0,0,1)$ (spacelike vector), then $\alpha(s)=x_3(s)$ for all $s \in I$. Moreover, $\alpha(s)$ is the Euclidean height function (with sign) for the vector $(0,0,1)$
    \end{enumerate}
\end{remark}

The  \eqref{c4p1}, \eqref{c4p2}  and \eqref{ob1} show that investigating  the soliton solutions to the CF on $\mathbb{S}_{1}^{2}$  is equivalent to studying the  solutions $\psi(s)=(\alpha(s),\tau(s), \eta(s))$ of the system 
\begin{equation}\label{s2}
	\left\{\begin{array}{lll}
		\alpha^{\prime}(s)=\tau(s),\\
		\tau^{\prime}(s)=-\epsilon\alpha(s)-\epsilon a \tau(s)\eta(s),\\
		\eta^{\prime}(s)=-\epsilon a\tau^2(s),
	\end{array}\right.
\end{equation}
for each constant  $a>0$ and initial condition $\psi(0) \in H\cup C\cup S$, where
\begin{equation}
	\label{h} 
	\begin{array}{l}
		H:=\{(\alpha,\tau,\eta) \in \mathbb{R}^{3}: \alpha^2+\epsilon\tau^{2}-\epsilon\eta^2=-1\},\\	
		C:=\{(\alpha,\tau,\eta) \in \mathbb{R}^{3}\setminus \{0\}: \alpha^2+\epsilon \tau^{2}-\epsilon\eta^2=0 \}, \\
		S:=\{(\alpha,\tau,\eta) \in \mathbb{R}^{3}: \alpha^2+ \epsilon\tau^{2}-\epsilon\eta^2=1\}. 
	\end{array}	
\end{equation}
These are disjoint sets and if the initial condition $\psi(0) \in H$ (respectively $C$ or $S$) then the solution $\psi(s)$,  defined on the maximal interval $I=(\omega_{-},\omega_{+})$, will be contained in $H$ (respectively $C$ or $S$) for all $s \in I$.

In our next Lemma, we study the solutions of \eqref{s2} when $\tau(s)$ is constant. 
\begin{lemma}
	\label{tri}
	Let $\psi(s)=(\alpha(s),\tau(s),\eta(s))$ be a non-null solution of \eqref{s2} defined on the maximal interval $I$, $a>0$,  and initial condition $\psi(0)\in H\cup C\cup S$. The function $\tau(s)=b \in \mathbb{R}$, for all $s \in I$  if and only if $b \in \{-1,0,1\}$ and $I=\mathbb{R}$. Moreover,
	\begin{itemize}
		\item[i)] if $b=0$, then $\psi(s)=(0,0,\pm 1)$ is a  singular solution of \eqref{s2} and $\psi(s) \in H$ (respectively $\psi(s) \in S$) for all $s \in \mathbb{R}$, when $X$ is a spacelike curve (respectively $X$ is a timelike space),
		\item[ii)] if $X$  is a spacelike curve and $b^{2}=1$, then $a=1$ and $\psi(s)=(\pm s+\alpha(0),\pm 1,-s \mp  \alpha(0)) \in S$ for all $s \in \mathbb{R}$.
	\end{itemize}
\end{lemma}
\begin{proof} 
	If $b=0$, it follows from \eqref{s2} that $\alpha(s)=0$ for all $s\in I$. Using \eqref{h}, we obtain $-\epsilon \eta^2(s)=\pm 1$. \\
	If $b \neq 0$, it follows from  \eqref{s2} that $\alpha(s)=-ab\eta(s)$ and $\eta(s)=-\epsilon ab^2 s+\eta(0)$ for all $s\in \mathbb{R}$. Using \eqref{h}, we conclude that 
	$$(a^2b^2-\epsilon)\eta^2(s)=\gamma -\epsilon b^2,$$
	where $\gamma \in \{-1,0,1\}$. Hence, the function $(a^2b^2-\epsilon)\eta^2(s)$ is constant for all, $s\in \mathbb{R}$, it follows that, 
	$a^2b^2=\epsilon \;\;\text{and} \;\; \epsilon b^2=\gamma$. Therefore, $\epsilon=1$, $\gamma=1$, $b=\pm 1$, $a=1$ and $\alpha (s)=\mp \eta (s)$ for all $s\in \mathbb{R}$. 
\end{proof}
The next lemma prove that $\omega_-$ and $\omega_{+}$ are finite, when $\alpha(s), \tau(s)$ and $\eta(s)$ are unbounded functions on $(\omega_{-},s_0)$ and $(s_0,\omega_+)$, $s_0 \in I$, respectively.
\begin{lemma}
	\label{p}
	Let $\psi(s)=(\alpha(s),\tau(s),\eta(s))$ be a non-trivial solution of \eqref{s2} defined on the maximal interval $I$, $a>0$, and initial condition $\psi(0)\in H\cup C\cup S$, where $H,C$ and $S$ are given by \eqref{h}. Suppose that $ \lim_{s \to \omega_{-}}|\tau(s)|=\lim_{s \to \omega_{-}}|\eta(s)|=\lim_{s \to \omega_{-}}|\alpha(s)|=+\infty$, then $\omega_{-}>+\infty$. Likewise, if $\lim_{s \to \omega_{+}}|\tau(s)|=\lim_{s \to \omega_{+}}|\alpha(s)|=\lim_{s \to \omega_{+}}|\eta(s)|=+\infty$, then $\omega_{+}<+\infty.$  
\end{lemma}
\begin{proof}
	Since, $\lim_{s \to \omega_{-}}|\alpha(s)|=+ \infty$, $\lim_{s \to \omega_{-}}|\eta(s)|=+\infty$ and $\alpha^2(s)+\epsilon\tau^{2}(s)-\epsilon\eta^2(s)=\delta$, $\delta \in \{-1,0,1\}$, we obtain
	$$\lim_{s \to \omega_{-}}\frac{\alpha(s)}{\eta(s)}=\lim_{s \to \omega_{-}}\frac{\alpha ^{\prime}(s)}{\eta{\prime}(s)}=\lim_{s \to \omega_{-}}\frac{1}{-\epsilon a\tau(s)}=0,\,\,\lim_{s \to \omega_{-}}\frac{\eta(s)}{\alpha(s)}=\pm\infty,$$ $$\lim_{s \to \omega_{-}} \frac{\tau^{2}(s)}{\eta^2(s)}=\lim_{s \to \omega_{-}} \left[1-\frac{\alpha^{2}(s)}{\epsilon\eta^2(s)}+\frac{\delta}{\epsilon\eta^{2}(s)}\right]=1$$ 
	and
	$$\lim_{s \to \omega_{-}}\frac{2\tau(s)}{\alpha^{2}(s)+1}=\lim_{s \to \omega_{-}} \frac{-2\epsilon\alpha(s)-2a\epsilon\tau(s)\eta(s)}{2\alpha(s)\tau(s)}=\lim_{s \to \omega_{-}} -\frac{\epsilon}{\tau(s)}-\frac{a\epsilon\eta(s)}{\alpha(s)}=\mp\infty.$$
	
	  If $\lim_{s\to \omega_{-}}\tau(s)=+\infty$, then there exists $s_0 \in I$ such that $2\tau(s)>\alpha^2(s)+1$ for all $s<s_0$ and
	$$-2\arctan(\alpha(s))+2\arctan(\alpha(s_0))=\int_{\alpha(s)}^{\alpha(s_0)}\frac{2d\alpha}{\alpha^2+1}>-s+s_0,$$ that is, $s>2\arctan(\alpha(s))-2\arctan(\alpha(s_0))+s_0>-\pi+2\arctan(\alpha(s_0))+s_0 $. If $\lim_{s\to \omega_{-}}\tau(s)=-\infty$  then there exists $s_0 \in I$ such that $2\tau(s)<-[\alpha^2(s)+1]$ for all $s<s_0$ and
	$$-\pi+2\arctan(\alpha(s_0))<-2\arctan(\alpha(s))+2\arctan(\alpha(s_0))=\int_{\alpha(s)}^{\alpha(s_0)}\frac{2d\alpha}{\alpha^2+1}<s-s_0.$$
	 As $\tau(s)=\alpha^{\prime}(s)$ and $-\pi\leq 2\arctan(\alpha(s))\leq \pi$. Therefore, $\omega_->-\infty.$ The proof that $\omega_{+}<+\infty$ can be done analogously.
\end{proof}

We shall study the solutions of the system of ordinary differential equations \eqref{s2}, firstly considering $\epsilon = -1$, where the system is associated with a timelike soliton solution to the FC, and subsequently $\epsilon = 1$, where the system is associated with a spacelike soliton solution to the FC.

\subsection{Behavior of timelike soliton solutions to the FC}

\hspace*{\parindent} We begin with a qualitative analysis of the system \eqref{s2} for $\epsilon = -1$, which corresponds to studying a timelike soliton solution to the FC. Suppose that $\displaystyle X: I \rightarrow \mathbb{S}_{1}^{2}$ is a  timelike curve parametrized by arc length $s$, such that $T(s)=(t_1(s),t_2(s), t_3(s))$ is the unit tangent vector field satisfying $-t_1^2+t_2^2+t_3^2=-1$, i.e, $t_1(s)\neq 0$ for all $s\in I$. The function $\tau(s)$ defined by \eqref{c4e5} has the following properties. If $e=(-1,0,0)$, timelike vector, then $\tau(s)=t_1(s)>1$ or $\tau(s)=t_1(s)<-1$ for all $s \in I$. If $e=(-1,1,0)$, lightlike vector, then $\tau(s)=t_1(s)+t_2(s)>0$ or $\tau(s)=t_1(s)+t_2(s)<0$ for all $s \in I$. 

 Note also that when $X(s)$ is a timelike soliton solution with $\tau(s)=\langle T(s), (-1,0,0)\rangle>1$ (respectively $\tau(s)=\langle T(s), (-1,1,0)\rangle>0$, then $X(-s)$ timelike soliton solution with $\tau(s)=-\langle T(-s), (-1,0,0)\rangle<-1$ (respectively $\tau(s)=-\langle T(-s), (-1,1,0)\rangle<0$. In this context, when $\psi(0) \in H\cup C$, we study only the case $\tau(s)>0$. 

 We will first classify the singular points of the system \eqref{s2} at $\epsilon=-1$ that lie in the set $S$.

\begin{lemma}\label{lem2}
	Let $\Phi:S\rightarrow TS\subset \mathbb{R}^3$ be the differential vector field given by
	$$\Phi(\alpha,\tau,\eta)=(\tau,\alpha+a\tau\eta,a\tau^2)$$
	where $a > 0$. Then $P = (0, 0, 1)$ and $-P = (0, 0,-1)$ are the singular points of $\Phi$ and the eigenvalues of $d\Phi_P$ and $d\Phi_{-P}$ are given respectively by
	$$\lambda_P=\frac{a\pm\sqrt{a^2+4}}{2},\;\;\lambda_{-P}=\frac{-a\pm\sqrt{a^2+4}}{2}.$$
\end{lemma}
\begin{proof}
	Note that if $\Phi(\alpha,\tau,\eta)=0$, then $\alpha=\tau=0$ and $\eta=\pm 1$. Hence, $P=(0,0,1)$ and $-P=(0,0,-1)$ are the singular points of $\Phi$. The tangent plane at each singular point is defined by $T_{\pm P}S\approx\{(w_1,w_2,w_3) \in \mathbb{R}^{3}: w_3=0\}$. Thus,
	\begin{equation*}
		d\Phi=\left(\begin{array}{ccc}
			0                       & 1          & 0  \\
			1           & a\eta  & a\tau \\
			0                       & 2a\tau &  0
		\end{array}\right).
	\end{equation*}
	The number $\lambda$ is an eigenvalue of $d\Phi_{\pm P}$ if there is a non null vector $w \in T_{\pm P} S$ such that $d\Phi_{\pm p}(w)=\lambda w$ i.e.
	\begin{equation*}
		\left(\begin{array}{ccc}
			0                       & 1          & 0  \\
			1                      & \pm a  & 0 \\
			0                       & 0 &  0
		\end{array}\right)\left(\begin{array}{lll}
			w_{1}\\
			w_{2}\\
			0                      
		\end{array}\right)=\lambda (w_{1},w_{2},0).
	\end{equation*}
	Therefore,
	\begin{equation*}
		\left\{\begin{array}{cc}
			w_{2}=\lambda w_{1},\\
			w_{1}\pm aw_{2}=\lambda w_{2}.
		\end{array}\right. 
	\end{equation*}
	Hence, $\lambda$ satisfy $\lambda^{2} \mp a\lambda-1=0$ and
	
	\begin{equation*} 
		\lambda_{P}=\frac{a\pm \sqrt{a^2+4}}{2}\hspace{0.5 cm} \text{e} \hspace{0.5 cm} \lambda_{-P}=\frac{-a\pm \sqrt{a^2+4}}{2}.
	\end{equation*}
\end{proof}

In Lemma \ref{tri}, we saw that $(0, 0,\pm1)$ are saddle points for the vector field $\Phi$ on S, i.e.,  $\psi(s) =(0, 0,\pm 1)$, $s\in I$ are singular solutions of \ref{s2}. If the functions $\alpha$ and $\tau$ are identically zero, then
the corresponding curve $X(s)$ is the intersection of the Sitter sphere with the plane going through the origin, orthogonal to $(0, 0, 1)$. Hence, both singular solutions of the system correspond to the same curve.

To study non-trivial solutions of the system \eqref{s2}, we consider the singular point $P = (0, 0, 1)$ and a solution $\psi(s, q)$ of the system with initial condition $q \in S$. Since the eigenvalues of the linearized system at the singular point are non-zero, it follows that the local behavior of the system \eqref{s2} is equivalent to that of the linearized system. Hence, there exist initial conditions $Q, \overline{Q} \in S \setminus \{P\}$ such that $\lim_{s \to +\infty} \psi(s, Q)=P$ and $\lim_{s \to -\infty} \psi(s, \overline{Q})=P$. We define the unstable and stable sets as
\begin{equation} \label{w}
	\displaystyle W^{i}(P)=\{Q \in S: \lim_{s \to -\infty}\psi(s,Q)=P\}\,\,\,\, \text{and}\,\,\,\,
	\displaystyle W^{e}(P)=\{Q\in S: \lim_{s \to +\infty}\psi(s,Q)=P\}.
\end{equation}
Observe that the set $C$ is not closed and that $(0,0,0)$ is a singular point of \eqref{s2}. Still,  this point is not a attractor or repulsor, because $\alpha(s)$ and $\eta(s)$ are increasing functions and taking $\alpha(s_0)>0$, $\eta(s_0)<0$, then $\lim_{s \to \omega_{+}}\alpha(s)\neq 0$ and $\lim_{s \to \omega_{-}}\eta(s)\neq 0$. The next lemma, we analyze the case where $\tau(s)$ does not have critical point and $\psi(0) \in C.$ 
\begin{lemma}\label{e3}
Let $\psi(s)=(\alpha(s),\tau(s),\eta(s))$ be a non-trivial solution of \eqref{s2} defined on the maximal interval $I$, $a>0$, $\epsilon=-1$ and the initial condition $\psi(0)\in C $, where $C$ is given by \eqref{h}. If $\tau(s)$ is monotone on $I$ then either $\lim_{s\to-\infty}\psi(s)=(0,0,0)$ or $\lim_{s\to+\infty}\psi(s)=(0,0,0)$.
\end{lemma}
\begin{proof}
	Initially, we consider $\tau(s)$ decreasing on $I$. As $\tau(s)>0$ for all $s \in I$, $\tau(s)$ is bounded on $(s_0, \omega_+)$, interval $I$ is maximal and $\alpha^{2}(s)+\eta^{2}(s)=\tau^{2}(s)$, then $\omega_{+} = +\infty$. By Barbalat's Lemma, we have $\lim_{s \to +\infty} \tau(s) = 0$. Consequently, we obtain $\lim_{s \to +\infty} \alpha(s) = \lim_{s \to +\infty} \eta(s) = 0$. If $\tau$ is strictly increasing, the proof follows analogously.
\end{proof}
From Lemma \eqref{e3}, we define the unstable and stable sets as
\begin{eqnarray} \label{w1}
	\displaystyle W^{i}(0)&=&\{Q \in C: \lim_{s \to -\infty}\psi(s,Q)=(0,0,0)\}\,\,\,\, \text{and}\,\,\,\, \nonumber\\
	\displaystyle W^{e}(0)&=&\{Q\in C: \lim_{s \to +\infty}\psi(s,Q)=(0,0,0)\}.
\end{eqnarray}

The next lemma provides details on the critical points of $\alpha(s)$ and $\eta(s)$ for a non-trivial solution of the system \eqref{s2} with $\epsilon=-1$. 

\begin{lemma}
	\label{lem25}
	Let $\psi(s)=(\alpha(s),\tau(s),\eta(s))$ be a non-trivial solution of \ref{s2} defined on the maximal interval $I$, $a>0$, $\epsilon=-1$ and the initial condition $\psi(0)\in H\cup C \cup S$, where $H$, $C$ and $S$ are given by \eqref{h}. Let $\psi(0) \in S$ be, if $s_0 \in I$ is a critical point of $\alpha(s)$, then $s_{0}$ is the global minimum (respectively global maximum) of $\alpha(s)$ if and only if $\alpha(s_{0})>0$ (respectively $\alpha(s_{0})<0$). If $s_0$ is a critical point of $\tau(s)$, then  $s_0$ is a global minimum (respectively global maximum) point of $\tau(s)$ if and only if $\tau(s_0)>0$ (respectively $\tau(s_0)<0$). Moreover, there are always $s_1,s_2$ such that the functions $\alpha(s)$ and $\tau(s)$ are monotone on the intervals $(\omega_{-},s_1)$ and $(s_2, \omega_{+}).$ 
\end{lemma}
\begin{proof} 
Initially, let $\psi(0) \in S$ be and $\in s_{0}$ be a critical point of $\alpha(s)$, then $\tau(s_0)=0$. It follows from \eqref{s2} that
	\begin{equation}\label{s_0}
		\alpha^{\prime \prime}(s_0)=\alpha(s_0).
	\end{equation}
	Note that, if $\alpha(s_0)=\tau(s_0)=0$, then $\eta(s_0)=\pm1$ and  $\psi(s)=(0,0,\pm1)$, that is, a trivial solution. Let $\alpha(s_0)>0$, thus $s_{0}$ is a local minimum point of $\alpha(s)$. If there is another critical point $s_1$ of $\alpha(s)$ such that $s_0$ and $s_1$ are consecutive, then $\alpha(s_1)>\alpha(s_0)>0$, because $s_0$ is  a local minimum point. Thus, $\alpha^{\prime \prime}(s_1)=\alpha(s_1)>0$ and $s_1$ is a local minimum point $\alpha(s)$, this is absurd. Therefore, if $\alpha(s_0)>0$, then $s_0$  is the global minimum point $\alpha(s)$.  The proof that $s_0$ is the global maximum point of $\alpha(s)$ if $\alpha(s_0)<0$ can be carried out in a manner analogous to the previous case.

Let $s_{0}$ be a critical point of $\tau(s)$, then $\alpha(s_0)+a\tau(s_0)\eta(s_0)=0$ and $\tau''(s_0) = \tau(s_0)[1 + a^2 \tau(s_0)^2]$. If $\tau(s_0)=0$, then $\alpha(s_0)=0$, $\eta(s_0)=\pm1$ and  $\psi(s)=(0,0,\pm1)$, that is, a trivial solution. Thus, $s_0$ local minimum (respectively maximum) point of $\tau(s)$ if and only if $\tau(s_0)>0$ (respectively $\tau(s_0)<0$). The proof that $s_0$ is the global minimum (respectively maximum) point of $\tau(s)$ can be carried out in a manner analogous to the case of $\alpha(s)$.

The functions $\alpha(s)$ and $\tau(s)$ have at most one critical point, always there exists $s_1$, $s_2$ such that $\alpha(s)$ and $\tau(s)$ is monotone on the intervals $(\omega_{-},s_1)$ and $(s_2, \omega_{+}).$

\end{proof}\
Notice that when $\psi(0) \in H \cup C$, the functions $\alpha(s)$ and $\eta(s)$ do not have critical points, and the function $\tau(s)$ does not have maximum point, because we are taking  $\tau(s)>0$ for all $s \in I$. By Lemma \ref{lem25}, any non-trivial solution of the system \eqref{s2} such that $\psi(0) \in W^{i}(P) \cup W^{e}(P)\cup W^{e}(0)\cup W^{i}(0)$ has no critical points for the functions $\alpha(s)$ and $\tau(s)$. For initial conditions in $H$, we present the following results.

\begin{lemma}
	\label{lem26}
	Let $\psi(s)=(\alpha(s),\tau(s),\eta(s))$ be a non-trivial solution of \eqref{s2} defined on the maximal interval $I=(\omega_{-},\omega_{+})$, $a>0$, $\epsilon=-1$  and initial condition $\psi(0)\in H$, where $H$ is given by  \eqref{h}. Then, there exists $s_0 \in I$ such that $s_0$ is a critical point of $\tau(s)$.
	\end{lemma}
\begin{proof}
Suppose by contradiction that $\tau(s)$ has no critical points; then $\tau$ is either strictly increasing or strictly decreasing on $I$. Assume that $\tau$ is strictly decreasing, hence $\tau(s)$ is bounded on $(\Bar{s}, \omega_+)$, for every $\bar{s} \in I$. By the equation $\alpha(s)^2 + \eta(s)^2 = \tau(s)^2 +\delta$, $s \in I$, we have that $\alpha$ and $\eta$ are monotonic and bounded on $(\bar{s}, \omega_+)$, so there exists $P$ such that $\lim_{s \rightarrow \omega_+} \psi(s) = P$, that is, $P$ is a singular point in $H$, this a contradiction, because by Lemma \ref{tri},  there no exists singular point on $H$. Analogously, one proves the case where $\tau(s)$ is increasing. 
\end{proof}

\begin{lemma}
	\label{un}
	Let $\psi(s)=(\alpha(s),\tau(s),\eta(s))$ be a non-trivial solution of \eqref{s2} defined on the maximal interval $I$, $a>0$, $\epsilon=-1$ and initial condition $\psi(0)\in H\cup C\cup S$. Let $s_0 \in I$. Consider $W^{i}(P)$, $W^{e}(P)$, $W^{i}(0)$ and $W^{e}(0)$ given by \eqref{w} and \eqref{w1}. If $\psi(0) \in (H\cup C\cup S)\setminus (W^{i}(P)\cup W^{i}(0))$ (respectively $\psi(0) \in (H\cup C\cup S)\setminus (W^{e}(P)\cup W^{e}(0)$), then $\omega_{-}>-\infty$, $ \lim_{s \to \omega_{-}}|\tau(s)|=\lim_{s \to \omega_{-}}-\eta(s)=\lim_{s \to \omega_{-}}|\alpha(s)|=+\infty$(respectively $\omega_{+}<+\infty$ $\lim_{s \to \omega_{+}}|\tau(s)|=\lim_{s \to \omega_{+}}|\alpha(s)|=\lim_{s \to \omega_{+}}\eta(s)=+\infty$).   
\end{lemma}
\begin{proof}		
	Let $s_1,$ $s_2$ be given by Lemma \eqref{lem25}, then the functions $\alpha(s)$, $\tau(s)$ and $\eta(s)$ are monotonic on $(s_2,\omega_{+})$ and $(\omega_{-},s_1)$. If $(H\cup C\cup S)\setminus (W^{e}(P)\cup W^{e}(0))$ assume by contradiction that $\tau(s)$ is limited in $(s_2,\omega_{+})$. As $\alpha^{2}(s)+\eta^{2}(s)=\tau^{2}(s)+\delta$, so the functions $\alpha(s)$ and $\eta(s)$ are limited and monotones in $(s_2,\omega_{+})$. The  $ \lim_{s \to \omega_{+}} \tau(s)$ exists, because $\tau^{2}(s)=\eta^{2}(s)+\alpha^{2}(s)-\delta$. Hence, there is $Q \in \mathbb{R}_{1}^{3}$ such that $ \lim_{s \to \omega_{+}}(\alpha(s),\tau(s),\eta(s))=Q$, $\omega_{+}=+\infty$ and $\bar{\psi}(s)=Q$ is a singular solution of \eqref{s2}, this a contradiction, because $\psi(0)\notin W^{e}(P)\cup W^{e}(0)$. Therefore, $ \lim_{s \to \omega_{+}}|\tau(s)|=+\infty$.
	
	Now, Suppose by contradiction that $\eta(s)$ is bounded on $(s_2,\omega_{+})$, thus using the Cauchy-Schwarz inequality to $s \in (s_2,\omega_{+})$, we obtain
	
	$$|\alpha(s)-\alpha(s_2)|\leq\int_{s_2}^s|\tau(s)|\cdot 1 ds\leq\left(\int_{s_2}^{s}\tau^{2}ds\right)^\frac{1}{2}\cdot \left([s-s_2]\right)^\frac{1}{2}\leq\left(\frac{1}{a}[\eta(s)-\eta(s_2)]\right)^\frac{1}{2}[\omega_{+}-s_2]^{\frac{1}{2}}$$	
	and, $\alpha(s)$ also bounded on $(s_2,\omega_{+})$, this contradicts the fact that $\tau(s)$ is unbounded and $\tau^{2}(s)=\alpha^2(s)+\eta^{2}(s)-\delta$. Therefore, $\eta(s)$ is unbounded on $(s_2,\omega_{+})$. 
	
	We proof that $\alpha(s)$ is unbounded on $(s_2, \omega_+)$. Observe that, if $\alpha(s)$ is bounded, \newline $\lim_{s \to \omega_{+}}\alpha(s)/\eta(s)=0$, $\lim_{s\to \omega_{+}}\eta(s)=\lim_{s\to \omega_{+}}|\tau(s)|=+\infty$ and there exists $s_2$ with $\eta(s_2)>1$ such that $\alpha^{2}(s)<\eta^2(s)$, $\eta^{2}(s)-\delta< 3\eta^{2}(s)$ and $\tau(s)\neq 0$ for all $s>s_2$. Consequently,  $\tau^{2}(s)=\alpha^{2}(s)+\eta^{2}(s)-\delta<4\eta^2(s)$, that is, $|\tau(s)|<2\eta(s)$ for all $s>s_2$. Note that
	$$\frac{d\alpha}{d\eta}=\frac{\tau(s)}{a\tau^{2}}=\frac{1}{a \tau(s)}.$$
	If $\tau(s)>0$ for $s>s_2$, then 
	$$\frac{d\alpha}{d\eta}>\frac{1}{2a \eta(s)}, \,\, \text{and}\,\, \alpha(s)-\alpha(s_2)>\frac{1}{2a}\ln\left(\frac{\eta(s)}{\eta(s_2)}\right).$$
	If $\tau(s)<0$ for $s>s_2$, then 
	$$\frac{d\alpha}{d\eta}<-\frac{1}{2a \eta(s)}, \,\, \text{and}\,\, \alpha(s)-\alpha(s_2)<-\frac{1}{2a}\ln\left(\frac{\eta(s)}{\eta(s_2)}\right).$$
	This a contradiction, because $\lim_{s \rightarrow \omega_+} \ln\eta(s)=+\infty.$ Therefore, $\lim_{s \to \omega_{+}}|\alpha(s)|=+\infty$ and by Lemma \ref{p}, $\omega_{+}<+\infty$
		 
By an analogous argument, one obtains the results on $(\omega_{-},s_1)$. 
\end{proof}

\noindent \bf Proof of Theorem \ref{c9} \rm. Without loss generality, we can consider $v=ae$, where $a>0$ and

\begin{eqnarray*}
		e=	\left\{\begin{array}{l}
			(-1,0,0)\hspace{0.3 cm} \text{if} \hspace{0.3 cm}v \hspace{0.3 cm} \text{is a timelike vector},\\
			(-1,1,0)\hspace{0.3 cm} \text{se} \hspace{0.3 cm} v \hspace{0.3 cm}  \text{is a lightlike vector },\\
			(0,0,1)\hspace{0.3 cm} \text{se} \hspace{0.3 cm} v \hspace{0.3 cm}  \text{is a spacelike vector.}
		\end{array} \right.
	\end{eqnarray*}	 
	Let $\psi(s)=(\alpha(s),\tau(s),\eta(s))$ be a solution of \eqref{s2} defined on the maximal interval $I=(\omega_{-},\omega_{+})$, $a>0$ and initial condition $\psi(0) \in \mathbb{R}^{3}$ satisfy 		 
	\begin{eqnarray*}
		\alpha^{2}(0)-\tau^{2}(0)+\eta^{2}(0)=	\left\{\begin{array}{l}
			-1\hspace{0.3 cm} \text{if} \hspace{0.3 cm}v \hspace{0.3 cm} \text{is a timelike vector},\\
			0 \hspace{0.3 cm} \text{if} \hspace{0.3 cm} v \hspace{0.3 cm}  \text{is a lightlike vector},\\
			1 \hspace{0.3 cm} \text{if} \hspace{0.3 cm} v \hspace{0.3 cm}  \text{is a spacelike vector.}
		\end{array} \right.
	\end{eqnarray*}
	By Proposition \ref{c4p2}, there exists a solution solution to the CF $X(s)$, of curvature $k(s)=a\tau(s)$, whole the trihedron   
	$\displaystyle \{X(s),T(s),N(s)\}$ satisfy $$\alpha(s)=\langle X(s),e\rangle,\hspace{0.3 cm} \tau(s)=\langle T(s),e\rangle\hspace{0.3 cm}\text{and} \hspace{0.3 cm} \eta(s)=\langle N(s),e\rangle.$$
	
	Thus, the initial conditions of \eqref{s2} that depend on two constants determine the soliton solution in each case. Therefore, for each fixed vector $v \in \mathbb{R}_{1}^{3}\setminus \{0\}$, there is a 2-parameter family of non-trivial soliton solutions to the CF in $\mathbb{S}^{2}_{1}.$

    Let $\psi(s)$ be a non-trivial solution. It follows from Lemma \ref{un} that if $\psi(0) \in H\cup C \cup S \setminus \left(W^{i}(P)\cup W^{e}(P)\cup W^{i}(0) \cup W^{e}(0)\right)$, then $I=(\omega_{-}, \omega_+)$ is a bounded interval and at each end, $k(s)=a\tau(s)$ is unbounded. From the properties of the sets $H$ and $C$, it follows that $\tau(s)\neq0$ for all $s \in I$, that is, $k(s)$ has no zeros. By Lemma \ref{lem25}, we have that $\tau(s)$ does not admit a maximum point if $\tau(s)>0$, and it does not admit a minimum point if $\tau(s)<0$, that is, $k(s)$ has at most one zero, and thus, taking $\psi(0) \in S$ such that $\tau(0)=0,$ we obtain that the function has a unique zero, $k(s)$ is unbounded at both ends and $\psi(s)$ does not belong to $W^{i}(P)\cup W^{e}(P)\cup W^{i}(0) \cup W^{e}(0)$. If $\psi(0) \in \left(W^{i}(P)\cap W^{e}(P)\right)\cup \left(W^{i}(0) \cap W^{e}(0)\right)$, then it follows from Lemmas \ref{lem2} and \ref{e3} that $\psi(s)$ is defined for all $s \in \mathbb{R}$, $\lim_{s \to +\infty}\tau(s)=\lim_{s \to -\infty}\tau(s)=0$, thus the information regarding the critical points of $\tau(s)$ also enables us to conclude that     $\psi(s)$ is a trivial solution. 
    
    Hence, it follows from Lemmas \ref{lem2}, \ref{e3} and \ref{un} that if $\psi(s)$ is non-trivial solution and $\psi(0) \in W^{i}(P)\cup W^{i}(0)$ (respectively, $\psi(0) \in W^{e}(P)\cup W^{e}(0)$), then $\omega_{-}=-\infty$, $\lim_{s \to-\infty}k(s)=0$, $\omega_+<+\infty$, $k(s)\neq0$ for all $s$ and $\lim_{s \to\omega_{+}}|k(s)|=+\infty$ (respectively, $\omega_{-}>-\infty$, $\lim_{s \to-\omega_{-}}|k(s)|=+\infty$, $\omega_{+}=+\infty$, $k(s)\neq0$ for all $s$ and $\lim_{s \to+\infty}k(s)=0$). Therefore, there exist no non-trivial complete soliton solutions.
    
\subsection{Behavior of spacelike soliton solutions to the FC}
    
\hspace*{\parindent} Now, we study the system \eqref{s2} when $\epsilon=1$; in this case, the curve associated with the system is spacelike on the De Sitter space. Similarly to previous case, taking $\displaystyle X: I \rightarrow \mathbb{S}_{1}^{2}$ is a  spacelike curve parametrized by arc length $s$, such that $N(s)=(n_1(s),n_2(s), n_3(s))$ is the unit normal vector field satisfying $-n_1^2+n_2^2+n_3^2=-1$, i.e, $n_1(s)\neq 0$ for all $s\in I$. The function $\eta(s)$ defined by \eqref{c4e5} has the following properties. If $e=(-1,0,0)$, timelike vector, then $\eta(s)=n_1(s)>1$ or $\eta(s)=t_1(s)<-1$  for all $s \in I$. If $e=(-1,1,0)$, lightlike vector, then $\eta(s)=n_1(s)+n_2(s)>0$ or $\eta(s)=n_1(s)+n_2(s)<0$ for all $s \in I$. Thus, when $X(s)$ is a spacelike soliton solution with $\eta(s)=\langle N(s), (-1,0,0)\rangle>1$ (respectively $\eta(s)=\langle N(s), (-1,1,0)\rangle>0$, then $X(-s)$ spacelike soliton solution with $\eta(s)=-\langle N(-s), (-1,0,0)\rangle<-1$ (respectively $\eta(s)=-\langle N(-s), (-1,1,0)\rangle<0.$ In this context, when $\psi(0) \in H\cup C$, we study only the case $\eta(s)>0$.

The next Lemma studies the singular points of \eqref{s2} when $\epsilon=1$.

\begin{lemma}\label{lem3}
	Let $\Phi:H\rightarrow TH\subset \mathbb{R}^3$ be the differential vector field given by
	$$\Phi(\alpha,\tau,\eta)=(\tau,-\alpha-a\tau\eta,-a\tau^2)$$
	where $a > 0$. Then $P = (0, 0, 1)$ and $-P = (0, 0,-1)$ are the singular points of $\Phi$ and the eigenvalues of $d\Phi_P$ and $d\Phi_{-P}$ are given respectively by
	$$\lambda_P=\frac{-a\pm\sqrt{a^2-4}}{2},\;\;\lambda_{-P}=\frac{a\pm\sqrt{a^2-4}}{2}.$$
\end{lemma}
\begin{proof}
	The proof is analogous to Lemma \ref{lem2}.
\end{proof}

The Lemma \ref{lem3}, shows that the singular solution $\psi(s)=(0,0,1)$ of system \eqref{s2}, $\epsilon=1$ given in the Lemma \ref{tri} and a local attractor, and attracts exponentially as $a\geq 2$ and spirals inward as $0<a<2$. The other case presents a local repeller and belongs to another connected component of $H$. The next lemma shows that the singular solutions of \eqref{s2}, $\psi(s)=(0,0,1)$ it is a global attractor, just like, $\psi(s)=(0,0,0)$ is a global attractor for the solutions of \eqref{s2}, $\epsilon=1$ in $C$.
\begin{lemma}
	\label{14}
	Let $\psi(s)=(\alpha(s),\tau(s),\eta(s))$ be a solution non-trivial of \eqref{s2} defined on the maximal interval $I=(\omega_{-},\omega_{+})$, $a>0$, $\epsilon=1$  and initial condition $\psi(0)\in C \cup H$, where $C$ and $H$ are given by \eqref{h}. Then $\omega_+=+\infty$ and
	$\lim_{s\rightarrow +\infty}\alpha(s)=\lim_{s\rightarrow +\infty}\tau(s)=0,\;\; \lim_{s\rightarrow +\infty}\eta(s)=-\delta$
	where $\delta=0,-1$.	
\end{lemma}
\begin{proof}
	First, observe that $\eta(s)>0$ is monotonic and bounded on $(\bar{s}, \omega_+)$ for all $\bar{s} \in I$. Therefore, from the equation $\alpha(s)^2 + \tau(s)^2 = \eta(s)^2 + \delta$, $s \in I$, we conclude that $\alpha(s)$ and $\tau(s)$ are also bounded on $(\bar{s}, \omega_+)$. Since $I$ is a maximal interval, it follows that $\omega_+ = \infty$ and $\lim_{s \rightarrow \infty} \eta(s)$ exists. Moreover, $\eta''(s)=-2a\tau(s)\tau^{\prime}(s)=2a\alpha(s)\tau(s)+2a^2\tau^{2}(s)\eta(s)$ is bounded on $(\bar{s}, \omega_+)$, and by Barbalat's lemma, we have $\lim_{s \rightarrow \infty} \eta'(s) = 0$. Consequently, $\lim_{s \rightarrow \infty} \tau(s) = 0$. Applying Barbalat's lemma for $\tau(s)$, we given that $\lim_{s \rightarrow \infty} \tau'(s) = 0$, and hence $\lim_{s \rightarrow \infty} \alpha(s) =\lim_{s \rightarrow \infty} -\tau^{\prime}(s)-a\tau(s)\eta(s)= 0$. Finally, from the equation $\alpha(s)^2 + \tau(s)^2 = \eta(s)^2 + \delta$, $s \in I$, we conclude that
	$
	\lim_{s \rightarrow \infty} \eta(s) = -\delta.$		
\end{proof}
The next lemma provides information on the critical points of $\alpha(s)$ and $\eta(s)$.
\begin{lemma}\label{l12}
	Let $\psi(s)=(\alpha(s),\tau(s),\eta(s))$ be a solution non-trivial  of \eqref{s2} defined on the maximal interval $I$, $a>0$, $\epsilon=1$ and initial condition $\psi(0)\in H\cup C\cup S$, where $H$, $C$ and $S$ are given by \eqref{h}. Then, the function $\eta(s)$ is strictly decreasing on $I$. Moreover, if $s_0 \in I$ is a critical point of $\alpha$, then $s_0$ is a local minimum point (respectively, a local maximum point) if and only if $\alpha(s_0) < 0$ (respectively, $\alpha(s_0) > 0$).
\end{lemma}
\begin{proof}
	From system \eqref{s2}, it follows directly that the function $\eta(s)$ is strictly decreasing on $I$. Let $s_0$ be a critical point of $\alpha(s)$. Suppose $\alpha'(s_0) = \tau(s_0) = 0$ and $\alpha''(s_0) = \tau'(s_0) = -\alpha(s_0) = 0$, so $s_0$ is also a critical point of $\tau$. It follows from $\alpha^2+\tau^2-\eta^{2}=\delta$, $\delta\in \{-1,0,1\}$, that $\psi(s) \in H$ for each $s$ and $\eta(s_0) = \pm 1$, which would imply $\psi(s) = (0,0,\pm 1)$ for all $s \in I$, contradicting the assumption that $\psi$ is a non-trivial solution. Therefore, $\alpha''(s_0)=- \alpha(s_0)\neq 0$ and $s_0$ is a local minimum point (respectively, a local maximum point) if and only if $\alpha(s_0) < 0$ ($\alpha(s_0) > 0$).	
\end{proof}

The study of the critical points of $\tau(s)$ depends on the sets $H$, $C$, and $S$.
\begin{lemma}\label{ll}
	Let $\psi(s)=(\alpha(s),\tau(s),\eta(s))$ be a solution non-trivial  of \eqref{s2} defined on the maximal interval $I$, $a>0$, $\epsilon=1$ and initial condition $\psi(0)\in H\cup C\cup S$, where $H$, $C$ and $S$ is given by \eqref{h}. If $s_0 \in I$ is a critical point of $\tau$ and $\psi(s_0)\in  H\cup C$, then $0<a|\tau(s_0)| < 1$, and $s_0$ is a local minimum point (respectively local maximum point) if and only if $\tau(s_0) < 0$ (respectively $\tau(s_0) > 0$).	If $s_0 \in I$ is a critical point of $\tau$ and  $\psi(s_0)\in  S$, then $s_0$ is a local minimum (respectively local maximum) if and only if $a\tau(s_0) > 1$ or $-1 < a\tau(s_0) < 0$ (respectively $a\tau(s_0) < -1$ or $0 < a\tau(s_0) < 1$).
	
\end{lemma}
\begin{proof} 
	Let $s_0 \in I$ be a critical point of $\tau$, that is, $\tau'(s_0) = -\alpha(s_0) - a \tau(s_0)\eta(s_0) = 0$ and 
	
	\begin{equation}\label{ii}
		\tau^{\prime \prime}(s_0) = \tau(s_0)[a^2 \tau(s_0)^2 - 1].	 	
	\end{equation}
	Substituting $\alpha(s_0) = -a \tau(s_0)\eta(s_0)$ into equation \eqref{h}, we obtain
	\begin{equation}\label{i}
		[a^2 \tau(s_0)^2 - 1] \eta(s_0)^2 = \delta - \tau(s_0)^2,
	\end{equation}
	where $\delta \in \{ -1, 0, 1\}$.
	We claim that $\tau(s_0)\neq0$ and $\tau^{\prime \prime}(s_0)\neq 0$. Suppose by contradiction that $\tau^{\prime \prime}(s_0)=0$, that is, $\tau(s_0)=0$ or $a^2 \tau(s_0)^2=1$. If $\tau(s_0)=0$, thus, $\tau(s_0)=\alpha(s_0)=0$, $\eta^{2}(s_0)=-\delta$ and $\psi(0)\in H$, that is, $\psi(s)$ is a trivial solution, contradicting the hypothesis. The another case implies that $\delta=1=\tau^{2}(s_0)=a$ and $\alpha(s_0)=\pm\eta(s_0)$, but, this trivial solution given by in Lemma \ref{tri}. 
	
	If $\delta  \in \{-1,0\}$, it follows from equation \eqref{i} that $a^2 \tau(s_0)^2 - 1 < 0$. Finally, from equation \eqref{ii}, we obtain the results.
\end{proof}

\begin{lemma} \label{p1}
	Let $\psi(s)=(\alpha(s),\tau(s),\eta(s))$ be a solution non-trivial  of \eqref{s2} defined on the maximal interval $I$, $a>0$, $\epsilon=1$ and initial condition $\psi(0)\in H\cup C\cup S$, where $H$, $C$ and $S$ are given by \eqref{h}. Let $s_0 \in I$. If $\tau(s)$ is bounded on $(\omega_{-},s_0)$, then $\omega_{-}=-\infty$, $\lim_{s \to -\infty}a\tau(s)=\pm1$ and $ \lim_{s \to -\infty}\eta(s)=\lim_{s \to -\infty}|\alpha(s)|=+\infty.$ If $\tau(s)$ is unbounded on $(\omega_{-},s_0)$, then $\omega_{-}>-\infty$ and $ \lim_{s \to \omega_{-}}|\tau(s)|=\lim_{s \to \omega_{-}}\eta(s)=\lim_{s \to \omega_{-}}|\alpha(s)|=+\infty.$ 
\end{lemma}
\begin{proof}
	Assume by contradiction that the decreasing function $\eta(s)$ is limited in $(\omega_{-},s_0)$.  As $\alpha^{2}(s)+\tau^{2}(s)=\eta^{2}(s)+\delta$, $\delta \in \{-1,0,1\}$, so the functions $\alpha(s)$ and $\tau(s)$ are limited on $(\omega_{-},s_0)$ and $ \lim_{s \to \omega_{-}} \eta(s)$ exists. Since $I$ is a maximal interval, $\omega_- = -\infty$. On the other hand, $\eta''$ is bounded on $(-\infty,s_0)$, and by Barbalat's lemma we have $\lim_{s \to \infty} \eta'(s) = 0$. Since, $\eta^{\prime}(s)=-a\tau^{2}(s)$, then $\lim_{s \to \infty} \tau(s) = 0$, and again by Barbalat's lemma, since $\tau''$ is bounded on $(\bar{s}, \infty)$, it follows that $\lim_{s \to \infty} \tau'(s) = 0$, that is, $\lim_{s \to \infty} \alpha(s) =\lim_{s \to -\infty}-\tau^{\prime}(s)-a\tau(s)\eta(s)=0$ and $\lim_{s \to -\infty}\eta^{2}(s)=\lim_{s \to -\infty}\alpha^{2}(s)+\tau^{2}(s)-\delta=-\delta$. This contradicts Lemma \ref{14}, because $(0,0,1) \in H$ and $(0,0,0)$ are global attractors for solutions in $H$ and $C$, respectively, and, in these sets, the function $\eta(s)$ is positive. Therefore, $\lim_{s \to \omega_{-}}\eta(s)=+\infty$.  
	
	Suppose that $\tau(s)$ is bounded on $(\omega_-, s_0)$ for some $s_0 \in I$, that is, $|\tau(s)|<k$ for all $s<s_0$ and $\alpha(s)$ is unbounded on $(\omega_-, s_0)$ because $\alpha^2(s)+\tau^{2}(s)=\eta^{2}(s)+\delta$, $\delta \in \{-1,0,1\}$. Thus, it follows from first equation of \eqref{s2} that $|\alpha(s)-\alpha(s_0)|<k|s-s_0|$ and $\omega_- = -\infty$. Moreover, 
	$$\lim_{s\to -\infty}\frac{\eta^2(s)}{\alpha^{2}(s)}=\lim_{s\to -\infty}1+\frac{\tau^2(s)-\delta}{\alpha^{2}(s)}=1,$$ 
    and $$ \pm 1=\lim_{s \to -\infty}\frac{-\eta(s)}{\alpha(s)}=\lim_{s \to -\infty}\frac{a\tau^{2}(s)}{\tau(s)}=\lim_{s \to -\infty}a\tau(s).$$	
	
	Suppose now that $\tau(s)$ is unbounded on $(\omega_-, s_0)$, we will prove that $\alpha(s)$ is unbounded on $(\omega_-,s_0)$. Note that
	$$\frac{d\eta}{d\alpha}=-a \tau(s).$$
	Suppose by contradiction that $\alpha(s)$ is bounded on $(\omega_-, s_0)$, then for all $s<s_0$ we have: if $\lim_{s \to \omega_{-}}\tau(s)=-\infty$, for some $s<s_0$  
	$$\frac{d\eta}{d\alpha}=-a\tau(s)>2\alpha, \,\, \text{and}\,\, -\eta(s)+\eta(s_0)>-\alpha^{2}(s)+\alpha^2(s_0);$$
	if $\lim_{s \to \omega_{-}}\tau(s)=+\infty$, for some $s<s_0$ and $\eta(s_0)>0$, we have $4\eta^{2}(s)=3\eta^{2}(s)-\delta+\alpha^{2}(s)+\tau^{2}(s)>\tau^{2}(s)$ and
	$$\frac{d\eta}{d\alpha}=-a\tau(s)>-2a\eta(s), \,\, \text{and}\,\, \ln\left(\frac{\eta(s_0)}{\eta(s)}\right)>2a\alpha(s)-2a\alpha(s_0).$$
	In both cases we obtain a contradiction, because $\lim_{s \rightarrow \omega_-}\eta(s)=+\infty.$ Therefore, $\lim_{s \to \omega_{-}}|\alpha(s)|=+\infty$ and it follows from Lemma \ref{p} that $\omega_{-}>-\infty.$ 
\end{proof}

The next lemma presents the behaviour of $\tau(s)$ on $(s_0,\omega_+)$ , $s_0 \in I$, in the case where $\psi(0) \in S$.
\begin{lemma} \label{p2}
	Let $\psi(s)=(\alpha(s),\tau(s),\eta(s))$ be a solution non-trivial  of \eqref{s2} defined on the maximal interval $I$, $a>0$, $\epsilon=1$ and initial condition $\psi(0)\in S$, where $S$ is given by \eqref{h}. Let $s_0 \in I$. If $\tau(s)$ is bounded on $(s_0, \omega_{+})$, then $\omega_{+}=+\infty$, $\lim_{s \to +\infty}a\tau(s)=\pm1$ and $ \lim_{s \to +\infty}-\eta(s)=\lim_{s \to +\infty}|\alpha(s)|=+\infty.$ If $\tau(s)$ is unbounded on $(s_0, \omega_{+})$, then $\omega_{+}<+\infty$ and $ \lim_{s \to \omega_{+}}|\tau(s)|=\lim_{s \to \omega_{+}}-\eta(s)=\lim_{s \to \omega_{+}}|\alpha(s)|=+\infty.$ 
\end{lemma}
\begin{proof}
	The case that $\tau(s)$ is bounded on $(s_0, \omega_{+})$ is proved analogue to Lemma \ref{p1}, as well as, to prove that $\eta(s)$ is unbounded. We will proof that $\alpha(s)$ is unbounded on $(s_0,\omega_+)$. Observe that if $\alpha(s)$ is bounded. for all $s>s_0$. Note that
	$$\frac{d\eta}{d\alpha}=-a \tau(s).$$
	Suppose by contradiction that $\alpha(s)$ is bounded on $(s_0,\omega_+)$. If $\lim_{s \to \omega_{+}}\tau(s)=-\infty$, for some $s>s_0$,  we have 
	$$\frac{d\eta}{d\alpha}=-a\tau(s)>-2\alpha, \,\, \text{and}\,\, -\alpha^{2}(s)+\alpha^2(s_0)<\eta(s)-\eta(s_0).$$
	If $\lim_{s \to \omega_{+}}\tau(s)=+\infty$, for some $s>s_0$ and $\eta(s_0)<0$,  we have $4\eta^{2}(s)=3\eta^{2}(s)-1+\alpha^{2}(s)+\tau^{2}(s)>\tau^{2}(s)$ and
	$$\frac{d\eta}{d\alpha}=-a\tau(s)>2a\eta(s), \,\, \text{and}\,\, \ln\left(\frac{\eta(s)}{\eta(s_0)}\right)<2a\alpha(s)-2a\alpha(s_0).$$
	In both cases we obtain a contradiction, because $\lim_{s \rightarrow \omega_+} \eta(s)=-\infty.$ Therefore, $\lim_{s \to \omega_{+}}|\alpha(s)|=+\infty$	
\end{proof}

We show in the next lemma that there exist solutions such that $I=\mathbb{R}$ and that if $a|\tau(s_0)|\geq1$, then the maximal interval cannot be $\mathbb{R}$.
\begin{lemma} \label{c}
	Let $\psi(s)=(\alpha(s),\tau(s),\eta(s))$ be a solution non-trivial  of \eqref{s2} defined on the maximal interval $I$, $a>0$, $\epsilon=1$, $s_0 \in I$ and initial condition $\psi(0)\in H\cup C$, where $H$ and $C$ are given by \eqref{h}. If $a|\tau(0)|\geq 1$ and $\psi(0) \in H\cup C$, then the function $\tau(s)$ is unbounded on $(\omega_{-},s_0)$, $s_0 \in I$, $\omega_{-}>-\infty$ and $ \lim_{s \to \omega_{-}}|\tau(s)|=\lim_{s \to \omega_{-}}\eta(s)=\lim_{s \to \omega_{-}}|\alpha(s)|=+\infty.$ Moreover, there exist solutions such that $I=\mathbb{R}$, $ \lim_{s \to -\infty} a\tau(s)=\pm1$, $\lim_{s \to -\infty}\eta(s)=\lim_{s \to -\infty}|\alpha(s)|=+\infty$.
\end{lemma}
\begin{proof}
Let $\psi(s)$ be such that $\psi(0)=(\alpha(0), \tau(0), \eta(0)) \in H\cup C$ and $a|\tau(0)|\geq 1$. We have the Lemma \ref{p1} that $\lim_{s \to +\infty}\tau(s)=0$ and by Lemma \ref{ll} that $\tau(s)$ dos not have critical point if $a|\tau(s)|\geq 1$, then $\tau(s)$  is monotone on $(\omega_{-}, 0)$ and $\lim_{s \to \omega_{-}}|\tau(s)|=+\infty.$ 

To prove the other part, we will use Wazewski's principle. For this, consider $\phi(s)=\psi(-s)$, $\overline{\alpha}(s)=\alpha(-s)$, $\overline{\tau}(s)=\tau(-s)$, $\overline{\eta}(s)=\eta(-s)$ and $\overline{\eta}(s)=\sqrt{\overline{\alpha}^2(s)+\overline{\tau}^2(s)-\delta}$, $\delta \in \{-1,0\}$. Thus, the system \eqref{s2} reduce to
\begin{equation}\label{21}
	\left\{\begin{array}{ll}
		\overline{\alpha}^{\prime}(s)=-\overline{\tau}(s),\\
		\overline{\tau}^{\prime}(s)=\overline{\alpha}(s) + a\overline{\tau}(s)\sqrt{\overline{\alpha}^2(s)+\overline{\tau}^2(s)-\delta},\\		
	\end{array}\right.
\end{equation}
and by Lemma \ref{14}, we have that $\lim_{s\to-\infty}\overline{\alpha}(s)=\lim_{s\to-\infty}\overline{\tau}(s)=0.$ Define the following set $W=\{(\overline{\alpha},\overline{\tau}) \in \mathbb{R}^{2}:\overline{\alpha}<0, 0<a\overline{\tau}<1\}$, $a>0$. The boundary $\partial W$ of $W$ is given by $\partial W= W_1 \cup W_2 \cup W_3,$ where $W_1=\{(\overline{\alpha},\overline{\tau}) \in \mathbb{R}^{2}:\overline{\alpha}<0, \overline{\tau}=0\},$ $W_2=\{(\overline{\alpha},\overline{\tau}) \in \mathbb{R}^{2}:\overline{\alpha}=0, 0\leq a\overline{\tau}\leq 1\}$ and $, W_3=\{(\overline{\alpha},\overline{\tau}) \in \mathbb{R}^{2}:\overline{\alpha}<0, a\overline{\tau}=1\}.$ As $(0,0)$ is a singular solution, we consider $a\tau(s)>0$ in $W_2.$

Define $V(s)=(\overline{\alpha}^{\prime}(s),\overline{\tau}^{\prime}(s))$. If $(\overline{\alpha}(0),\overline{\tau}(0))$ belong to $W_1$, then $V(0)=(0,\overline{\alpha}(0))$ points outward from $W$ in $W_1$ since $\overline{\alpha}(0)<0$. If $(\overline{\alpha}(0),\overline{\tau}(0))$ belong to $ W_2$, then $V(0)=(-\overline{\tau}(0),a\overline{\tau}(0)\sqrt{\overline{\tau}^2(0)-\delta})$ points inward from $W$ in $W_2$ since $\overline{\tau}(0)>0$ and $\delta \in \{-1,0\}$. If $(\overline{\alpha}(0),\overline{\tau}(0))$ belong to  $W_3$, then $V(0)=(-a^{-1},\overline{\alpha}(0)+\sqrt{\overline{\alpha}^{2}(0)+a^{-2}-\delta})$ points outward from $W$ in $W_3$ since $\delta \in \{-1,0\}$. The diagram in Figure \ref{f4} illustrates the field $V$ on the boundary of $W$.
\begin{figure}[h!]
\centering
\includegraphics[height=3 cm]{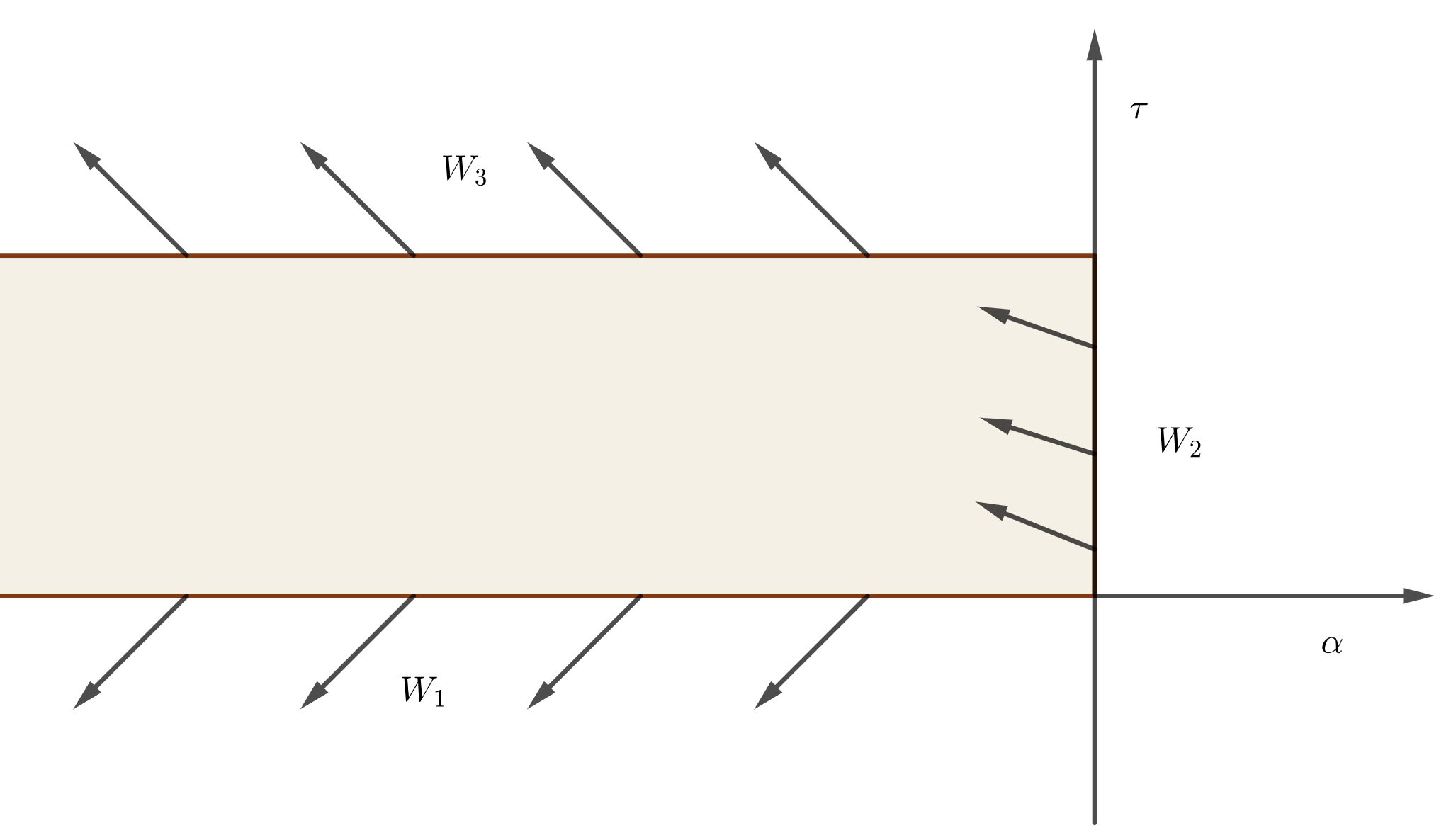}
\caption{Representation of the field $V$ on the boundary of $W$}
\label{f4}
\end{figure}

Thus, the set of strict egress points coincides with the set of egress points. Let $W^{-}$ be the set of strict egress points. Taking $P_1=(\overline{\alpha}_0,0)$ belong to $W_1$, $P_2=(\overline{\alpha}_0,a^{-1})$ belong to $W_3$, $\overline{\alpha}_0<0$ and consider the set $Z=[P_1,P_2]=\{P_1(1-t)+tP_2; 0\leq t\leq 1\} \subset (W \cup \partial W).$ Hence, $W^{-}\cap Z=\{P_1, P_2\}$ and $W^{-}\cap Z$ is a retract of $W^{-}$, but not a retract of $Z$, because $Z$ is a connected set. Therefore, it follows from Wazewski's principle that there exists $(\overline{\alpha}_0, \overline{\tau}_0) \in Z\setminus \{P_1, P_2\}$ such that the solution $\phi(s)=(\overline{\alpha}(s), \overline{\tau}(s))$ with $\phi(0)=(\overline{\alpha}_0, \overline{\tau}_0)$ belong to $W$ for each $s>0,$ that is, $0<\overline{\tau}(s)<a^{-1}$. Thus, it follows \ref{p1} that $\omega_+=+\infty$, $\lim_{s \to +\infty}a\overline{\tau}(s)=1$ and then $\lim_{s \to -\infty}a\tau(s)=1$. 

Considering $\overline{W}=\{(\overline{\alpha},\overline{\tau}) \in \mathbb{R}^{2}:\overline{\alpha}>0, -1<a\overline{\tau}<0\}$, $a>0$, the argument is analogous for proving that there exist solutions with $\lim_{s \to -\infty}a\tau(s)=-1$.

\end{proof}
The next Lemma presents the behavior of the function when $a|\tau(0)|>1.$

\begin{lemma} \label{p31}
	Let $\psi(s)=(\alpha(s),\tau(s),\eta(s))$ be a solution non-trivial  of \eqref{s2} defined on the maximal interval $I$, $a>0$, $a\neq 1$, $\epsilon=1$, $s_0 \in I$ and initial condition $\psi(0)\in  S$, where $S$ is given by \eqref{h}. If $a|\tau(0)|\geq 1$, then $\tau(s)$ is unbounded on $(\omega_{-},0)$, $\omega_{-}>-\infty$ and $ \lim_{s \to \omega_{-}}|\tau(s)|=\lim_{s \to \omega_{-}}\eta(s)=\lim_{s \to \omega_{-}}|\alpha(s)|=+\infty$,  $\omega_{-}>-\infty$ or $\tau(s)$ is unbounded on $(0,\omega_{+})$ and \newline $ \lim_{s \to \omega_{+}}|\tau(s)|=\lim_{s \to \omega_{+}}-\eta(s)=\lim_{s \to \omega_{+}}|\alpha(s)|=+\infty.$ Moreover, there exist solutions such that $\tau(s)\neq 0$ for all $s \in I$ and it is unbounded on $(\omega_{-},0)$ and $(0, \omega_{+})$.
\end{lemma}
\begin{proof}
 From Lemma \ref{ll}, if $\psi(0)\in S$ and $s_0$ a critical point of $\tau(s)$ with $a|\tau(0)|\geq 1$, then $s_0$ is minimum local if $a\tau(s_0)> 1$ and it is maximum local if $a\tau(_0)< -1$. Thus, if $a\tau(0)\geq 1$ and $\psi(0) \in S$ and $\tau(s)$ is decreasing at $s=0$, then $\tau(s)$  is monotone on $(\omega_{-}, 0)$ and $\lim_{s \to \omega_{-}}\tau(s)=+\infty$; if $a\tau(0)\geq 1$ and $\tau(s)$ is increasing at $s=s_0$, then $\tau(s)$  is monotone on $(0, \omega_{+})$ and $\lim_{s \to \omega_{+}}\tau(s)=+\infty$. A similar situation occurs when $a\tau(0)\leq -1$. Further, it follows from Lemma \ref{p1} that $\alpha(s)$, $\eta(s)$ are unbounded when $\tau(s)$ is unbounded. 

 To prove the last part, take $\psi(0)\in S$ such that $\alpha(0)=-a\eta(0)\tau(0)$ and $a\tau(0)>1$, then $s=0$ is local minimum point for $\tau(s)$ and by Lemma \ref{ll}, $s=0$ is the global minimum point. Therefore, $a\tau(s)>1$ for all $s$, $\tau(s)$ is decreasing on $(\omega_{-},0)$ and increasing on $(0, \omega_{+})$. Using the first part of the lemma, we have that $\lim_{s \to \omega_{-}}\tau(s)=+\infty$ and $\lim_{s \to \omega_{+}}\tau(s)=+\infty$.  
\end{proof}
Wazewski's principle can also be used to study solutions in the set $S.$  
\begin{lemma} \label{c1}
	Let $\psi(s)=(\alpha(s),\tau(s),\eta(s))$ be a solution non-trivial  of \eqref{s2} defined on the maximal interval $I$, $a>0$, $\epsilon=1$, $s_0 \in I$ and initial condition $\psi(0)\in S$, where $S$ is given by \eqref{h}. There exist solutions such that $\omega_{-}=-\infty$ and $\lim_{s \to -\infty}a\tau(s)=\pm 1$. There also exist solutions such that $\omega_{-}=-\infty$ and $\lim_{s \to +\infty}a\tau(s)=\pm 1$.
\end{lemma}
\begin{proof}
In this case, the function $\eta(s)$ can be any real number. As $\eta(s)$ is decreasing, consider initially $\eta(s)=\sqrt{\alpha^2(s)+\tau^2(s)-1}$. To establish the existence of solutions such that $\omega_{-}=-\infty$ and $\lim_{s \to -\infty}a\tau(s)=1$, it is enough to consider the system \ref{21} with $\delta=1$, $W=\{(\overline{\alpha},\overline{\tau}) \in \mathbb{R}^{2}:\overline{\alpha}<-1, 0<a\overline{\tau}<1\}$, $0<a<1$ and repeat the argument from the proof of Lemma \ref{c} in order to apply Wazewski's principle. Likewise, consider  $W^1=\{(\overline{\alpha},\overline{\tau}) \in \mathbb{R}^{2}:\overline{\alpha}>1, -1<a\overline{\tau}<0\}$, $0<a<1$ to guarantee the existence of a solution such that $\lim_{s \to -\infty}a\tau(s)=-1$.

For the other cases, take $\eta(s)=-\sqrt{\alpha^2(s)+\tau^2(s)-1}$, the system \eqref{s2} reduces to
\begin{equation*}
	\left\{\begin{array}{ll}
		\alpha^{\prime}(s)=\tau(s),\\
		\tau^{\prime}(s)=-\alpha(s) + a\tau(s)\sqrt{\alpha^2(s)+\tau^2(s)-1}.\\		
	\end{array}\right.
\end{equation*}
Consider $0<a<1$, and apply Wazewski's principle to the sets $W^2=\{(\alpha,\tau) \in \mathbb{R}^{2}:\alpha>1, 0<a\overline{\tau}<1\}$ and $W^3=\{(\alpha,\tau) \in \mathbb{R}^{2}:\alpha<-1, -1<a\overline{\tau}<0\}.$
\end{proof}

The Lemmas \ref{c} and \ref{p31} ensures the existence of solutions such that the function $\tau(s)$ is unbounded at least one of the ends and that if $\psi(s)$ is non-trivial solution with $I=\mathbb{R}$, then $-1<a\tau(s)<1$ for all $s \in \mathbb{R}$. In the next lemma, we show that if $I=\mathbb{R}$ and $\psi(0) \in C\cup S$, then $\tau(s)$ has at least one zero.
\begin{lemma} \label{c3}
	Let $\psi(s)=(\alpha(s),\tau(s),\eta(s))$ be a solution non-trivial  of \eqref{s2} defined on the maximal interval $I$, $a>0$, $\epsilon=1$, $s_0 \in I$ and initial condition $\psi(0)\in  C\cup S$, where $C$ and $S$ are given by \eqref{h}. If $I=\mathbb{R}$, then $\tau(s)$ has at least one zero.
\end{lemma}
\begin{proof}
If $I=\mathbb{R}$, then by Lemmas \ref{c} and \ref{p31} we have that $-1<a\tau(s)<1$ for all $s \in \mathbb{R}$. 

Suppose that $\psi(0) \in C$ and $I=\mathbb{R}$, then $\alpha^2(s)-\eta^2(s)=[\alpha(s)+\eta(s)][\alpha(s)-\eta(s)]=-\tau^{2}$, $\alpha^2(s)<\eta^{2}(s)$, and by Lemma \ref{p1} we have that $\lim_{s \to -\infty}a\tau(s)=\pm 1$, $\lim_{s \to -\infty}|\alpha(s)|=\lim_{s \to -\infty}\eta(s)=+\infty$ and $\lim_{s \to +\infty}\alpha(s)=\lim_{s \to +\infty}\tau(s)=\lim_{s \to +\infty}\eta(s)=0$. If  $\lim_{s \to -\infty}a\tau(s)= 1$, then $\lim_{s \to -\infty}-\alpha(s)=\lim_{s \to -\infty}\eta(s)=+\infty$, $\lim_{s \to -\infty}\alpha(s)-\eta(s)=-\infty$, $f(s)=\alpha(s)+\eta(s)>0$, $\lim_{s \to \pm\infty}f(s)=0$. Hence, the function $f(s)$ has at the least one local maximum point $s_0 \in \mathbb{R}$, that is, $f^{\prime}(s_0)=0$. As $f^{\prime}(s)=\tau(s)[1-a\tau(s)]$ and $-1<a\tau(s)<1$ for all $s \in \mathbb{R}$, then there exist $s_0$ such that $\tau(s_0)=0$. In an analogous manner, taking $\lim_{s \to -\infty}a\tau(s)= -1$, we obtain that $g(s)=\alpha(s)-\eta(s)<0$, $\lim_{s \to \pm\infty}g(s)=0$, that is, there exist $s_0$ such that $\tau(s_0)=0$.  

Take $\psi(0) \in S$, $I=\mathbb{R}$. It follows from \ref{ll} that $\tau(s)$ does not have a local minimum point if $0<a\tau(s)<1$ and does not have a local maximum point if $-1<a\tau(s)<0$; then $\tau(s)$ has at least one zero.
\end{proof}
 Thus, if there exist solutions with $I =\mathbb{R}$ and whose curvature function has no zeros, then $\psi(0) \in H$ and $a \geq 2$, since $0 < a < 2$ the solutions spirals to singular solution $\psi(s) = (0, 0, 1)$, $s \in \mathbb{R}.$

\textbf{Proof of Theorem \ref{c8}.} The proof of the first part is similar to the proof of the first part of Theorem \ref{c9}. Thus, consider $v \in \mathbb{R}_1^3\setminus \{0\}$ and the maximal solution $\psi(s)=(\alpha(s),\tau(s), \eta(s))$ corresponding to the soliton solution of the CF $X(s)$, $s\in I$ of the curvature function $k(s)=a\tau(s).$ 

In general, the Lemmas \ref{14}, \ref{p1} and \ref{p2} conclude that the curvature function at each end is either unbounded or tends to one of the following constants $\{-1,0,1\}$.

It follows from Lemma \ref{14} that if $\psi(0) \in H\cup C$, then $\omega_+=+\infty$ and $\lim_{s \to +\infty}k(s)=0$. From Lemma \ref{lem3}, we have that if $|v|<2$ and $\psi(0) \in H$, the function $\tau(s)$ has infinitely many zeros, and if $|v|\geq 2$, it has at most a finite number of zeros. It follows from Lemma \ref{c} that there exist complete spacelike soliton solutions. By Lemmas \ref{c} and \ref{p31}, there exist soliton solutions whose curvature function is unbounded on at least one end. Moreover, by Lemma \ref{p31}, there exist soliton solutions whose curvature function has no zeros and is unbounded on each end.

In the next section, we obtain results for soliton solutions to the ICF using information from the soliton solutions to the CF on $\mathbb{S}^{2}_{1}$ and $\mathbb{H}^{2}$. For the study of the inverse flow on De Sitter space, using the ordinary differential equations directly to list which curves can be soliton solutions, see the article. \cite{TIAN}.

\subsection{Soliton solutions to the inverse curvature flow on the 2-dimensional De Sitter space and hyperbolic space}

An application of Theorem \ref{c9} is the study of the timelike soliton solutions to the inverse curvature flow on De Sitter space.
\begin{theorem}\label{f1}
There is a 2-parameter family of non-trivial timelike soliton solutions to the inverse curvature flow on the 2-dimensional De Sitter space. There exist no non-trivial timelike complete soliton solutions. The curvature function either tends to zero at both ends, or it is unbounded at one end and converges to zero at the other. Moreover, when the curvature function tends to zero at both ends, then it is bounded.
\end{theorem}
\begin{proof}
It follows from Proposition \ref{p11} that $X(s)$, $k(s)\neq 0$, for all $s$, is a timelike soliton to the CF if and only if the normal position is a timelike soliton to the ICF. Thus, by Theorem \ref{c9}, there is a 2-parameter family of non-trivial timelike soliton solutions to the inverse curvature flow on the 2-dimensional De Sitter space. Also, by Theorem \ref{c9}, if $X(s)$, $s\in I=(\omega_{-}, \omega_{+})$ is a maximal timelike soliton to the CF, then the curvature function has at most one zero. Taking, $k(s_0)=0$, then $\gamma(s)=N(s)$, $s \in (\omega_{-},s_0)$ and $\overline{\gamma}(s)=N(s)$, $s \in (s_0, \omega_{+})$ are soliton solutions to the ICF. 

Let $Y(s)$, $s \in I$, be a timelike maximal soliton solution to the ICF on the De Sitter space of curvature $k(s)\neq 0$. The normal vector field, $N_{Y}(s)$, is a timelike curve on the De Sitter space of curvature $1/k(s)$, and by Proposition \ref{p11} is a timelike soliton solution to the CF. It follows from Theorem \ref{c9} that the interval $I$ cannot be extended to $\mathbb{R}$. Moreover, if $1/k(s)$
 is unbounded at both ends, then the curvature is bounded and tends to zero at both ends. Conversely, if $1/k(s)$ tends to zero at one end and is unbounded at the other, then the curvature is unbounded at the corresponding end and tends to zero at the other.
\end{proof}

An application of Theorem \ref{c8} is the study of the soliton solutions to the inverse curvature flow on hyperbolic space.

\begin{theorem}\label{f2}
There is a 2-parameter family of non-trivial soliton solutions to the inverse curvature flow on the 2-dimensional hyperbolic space. At each end, the curvature function is either unbounded or tends to one of the following constants $\{-1,0,1\}$. In particular, there exist non-trivial soliton solutions whose curvature function is unbounded or tends to zero at both ends.
\end{theorem}
\begin{proof}
It follows from Proposition \ref{p11} that $X(s)$ is a soliton solution to the ICF on the 2-dimensional hyperbolic space if and only if the normal position, $N(s)$, is a soliton solution to the CF on the 2-dimensional De Sitter space. Thus, by Theorem \ref{c8}, there is a 2-parameter family of non-trivial soliton solutions to the inverse curvature flow on the 2-dimensional hyperbolic space. Taking $\overline{N}(s)$ the maximal solution to the CF on De Sitter space, it follows from Theorem \ref{c8} that the curvature function can have no zeros, finitely many zeros, or infinitely many zeros, that is, $\overline{N}(s)$ can be associated an infinite number of soliton solutions to the ICF on 2-dimensional hyperbolic space. Thus, the results follow.
\end{proof}

Observe that, if there exists a non-trivial soliton solution to the curvature flow on $\mathbb{S}^{2}_{1}$ with curvature that never vanishes, then there exists a complete soliton solution to the ICF on the 2-dimensional hyperbolic space. 

 Similarly, we can obtain information about the ICF of spacelike curves in De Sitter Space using the theorem describing soliton solutions to the curve shortening flow on $\mathbb{H}^{2}$, Theorem 2.3 of \cite{Silva}. This theorem showed that: \it For any vector $v \in \mathbb{R}^{3}_1$, there is a 2-parameter family of non-trivial soliton solutions to the curve shortening flow on $\mathbb{H}^{2}$. There are three classes of soliton curves on $\mathbb{H}^{2}$, determined by the type of the vector $v$. Each soliton solution is an embedded curve $X(s)$ on $\mathbb{H}^{2}$, defined for all $s \in \mathbb{R}$. Moreover, at each end, the curvature function $k(s)$ tends to one of the constants $\{-1,0,1\}$. \rm Thus, if $X(s)$, $s \in I=(\gamma_{-}, \gamma_+)$ is a spacelike soliton solution to the ICF on the De Sitter space of the curvature $k(s)\neq 0$, then the normal position, $N(s)$, is a curve on the 2-dimensional hyperbolic space of the curvature $1/k(s)$. By Proposition \ref{c3t2}, $N(s)$ is a soliton solution to the curve shortening flow. Thus, at each end, the curvature function of $X(s)$ is either unbounded or tends to one of the following constants $\{-1,1\}$. If $\lim_{s \to \gamma_{-}} k(s)=\pm 1$(respectively $\lim_{s \to \gamma_{+}} k(s)=\pm 1$), then $\gamma_{-}=-\infty$ (respectively, $\gamma_{+}=+\infty$). Moreover, if there exists a non-trivial soliton solution to the curve shortening flow on $\mathbb{H}^{2}$ with curvature that never vanishes, then there exists a complete spacelike soliton solution to the ICF on the De Sitter space. 
\section{Visualizing some soliton solutions to the CF and ICF}
In this section, we visualize some examples of soliton solutions to the CF on the 2-dimensional De Sitter space and soliton solutions to the ICF on the 2-dimensional hyperbolic space and De Sitter space. For this, consider
$$X(s)=\chi(\theta(s),\rho(s))=(\sinh(\rho(s)), \cos(\theta(s))\cosh(\rho(s)),\sin(\theta(s))\cosh(\rho(s)))$$
a timelike or spacelike curve parametrized by arc length $s$ and the curvature function $k(s)$. Thus,
\begin{eqnarray*}
T(s)&=&(
\cosh(\rho(s)) \, \rho'(s), \;
\cos(\theta(s)) \sinh(\rho(s)) \, \rho'(s) - \cosh(\rho(s)) \sin(\theta(s)) \, \theta'(s), \\
&&
\sin(\theta(s)) \sinh(\rho(s)) \, \rho'(s) + \cos(\theta(s)) \cosh(\rho(s)) \, \theta'(s)),
\end{eqnarray*}
$-\left[ \rho^{\prime}(s)\right]^2+\left[\theta^{\prime}(s)\cosh(\rho(s)\right]^2=\epsilon$, $\epsilon \in \{-1,1\}$
and
\begin{eqnarray*}
N(s)&=& \left(-\cosh^2(\rho(s)) \, \theta'(s),
\sin(\theta(s)) \, \rho'(s) - \cos(\theta(s)) \cosh(\rho(s)) \sinh(\rho(s)) \, \theta'(s),\right.\\
&&\left.-\cos(\theta(s)) \, \rho'(s) - \cosh(\rho(s)) \sin(\theta(s)) \sinh(\rho(s)) \, \theta'(s)\right).
\end{eqnarray*}
If $\epsilon=1$ and $k(s)\neq0$, then $N(s)$ is curve on the 2-dimensional hyperbolic space. Analogously, if $\epsilon=-1$ and $k(s)\neq0$, then $N(s)$ is a curve on the 2-dimensional De Sitter space. Moreover, $X(s)$ satisfies the following equation.
\begin{equation}\label{p4}
\left\{\begin{array}{ll}
	\cosh(\rho(s))\left[ \rho''(s)\theta'(s) - \rho'(s)\theta''(s) \right] 
+ \sinh(\rho(s))\,\theta'(s)\left[ -(\rho'(s))^2+\epsilon \right]= k(s),\\
-\left[ \rho^{\prime}(s)\right]^2+\left[\theta^{\prime}(s)\cosh(\rho(s)\right]^2=\epsilon.
\end{array} \right.
\end{equation}
The following figures are constructed from the numerical solutions of the system \eqref{p4} and Theorem \ref{c3t2}, that is, $X(s)$ is spacelike or timelike soliton solutions to the CF on the $\mathbb{S}^{2}_{1}$ if and only if $k(s)=\langle T(s),v\rangle$, where $v \in \mathbb{R}^{3}_{1}\setminus \{0\}$. We also observe that to construct graphs in two-dimensional hyperbolic space, we consider the hyperboloid model.
\begin{figure}[ht]
	\centering	
	\subfloat[]{\includegraphics[height=2.5 cm]{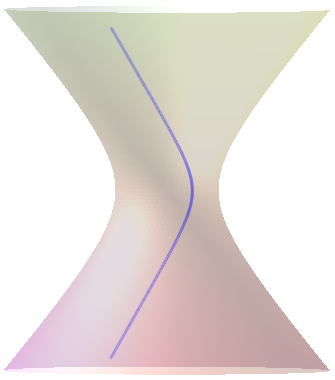}
\label{f1a}
} 
\quad
\subfloat[]{
\includegraphics[height=2.5 cm]{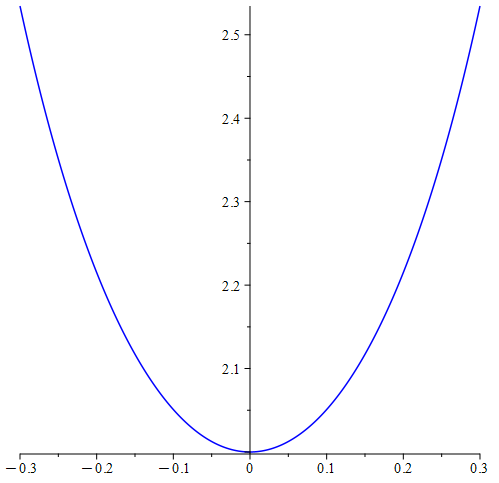}
\label{f1b}
}
\quad	
	\subfloat[]{
		\includegraphics[height=2.5 cm]{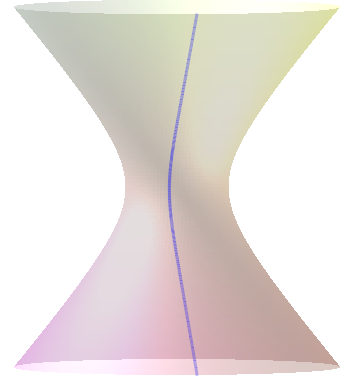}
\label{f1c}
	}
	\caption{Timelike soliton solution to the CF and ICF on $\mathbb{S}^{2}_{1}$ with $a=2$, $v=(1,0,0)$.}
	\label{fig10}
\end{figure} 

As shown in Theorem \ref{c9}, the curvature function of the timelike solution to the CF on the De Sitter space has at most one zero. 

In Figure \ref{f1a}, we present a timelike soliton solution to the CF with $k(s)>0$ for all $s \in I$. In Figure \ref{f1b}, we present the graph of the curvature function, and in Figure \ref{f1c}, the timelike soliton solution to the ICF on De Sitter space associated with the timelike soliton of Figure \ref{f1a}.

In Figure \ref{f2a}, we present a timelike soliton to the CF with $k(s_0)=0$ for some $s_0 \in I$. In Figure \ref{f2b}, we present the graph of the curvature function, and in Figure \ref{f2c}, the timelike soliton solutions to the ICF on De Sitter space associated with the timelike soliton of Figure \ref{f2a}. As $k(s_0)=0$ for some $s_0 \in I$, we have the two timelike soliton solutions to the ICF associated with the timelike soliton of Figure \ref{f2a}.

\begin{figure}[ht]
	\centering	
	\subfloat[]{\includegraphics[height=2.5 cm]{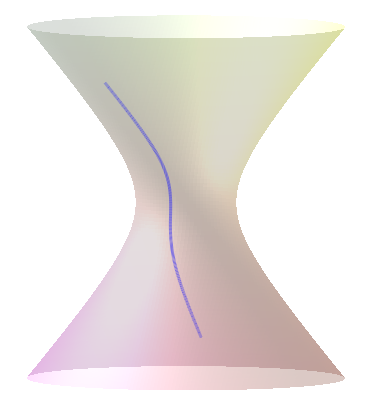}
\label{f2a}
} 
\quad
\subfloat[]{
\includegraphics[height=2.5 cm]{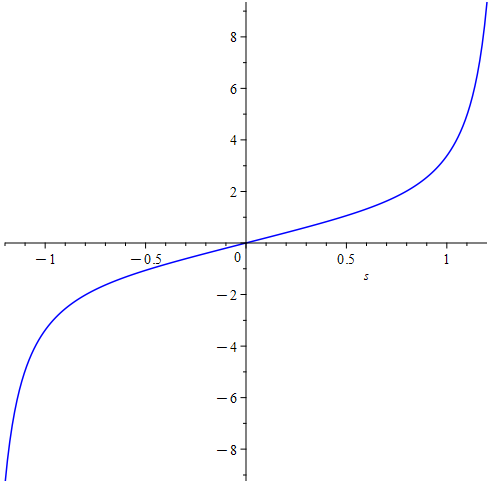}
\label{f2b}
}
\quad	
	\subfloat[]{
		\includegraphics[height=2.5 cm]{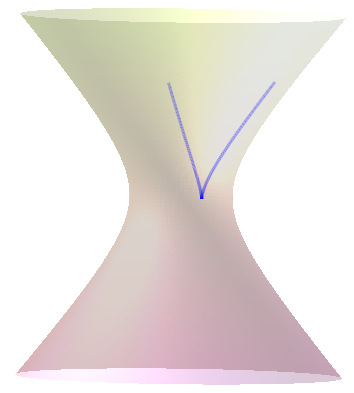}
\label{f2c}
	}
	\caption{Timelike soliton solution to the CF and ICF on $\mathbb{S}^{2}_{1}$ with $a=2$, $v=(0,0,1)$.}
	\label{fig1}
\end{figure} 
In Theorem \ref{c8}, the curvature function at each end is either unbounded or tends to one of the following constants $\{-1,0,1\}$. The curvature function can have no zeros, finitely many zeros, or infinitely many zeros. There exist non-trivial spacelike complete soliton solutions. Moreover, there exist non-trivial spacelike soliton solutions such that the curvature function has no zeros and is unbounded at each end. We do not present graphs of complete non-trivial solutions because determining the initial conditions is difficult.

In Figure \ref{f3a}, we present a spacelike soliton to the CF such that the curvature has infinitely many zeros, that is, $v$ is a timelike vector and $0<a<2$. In Figure \ref{f3b}, we present the graph of the curvature function, and in Figure \ref{f3c}, we present some soliton solutions to the ICF on the 2-dimensional hyperbolic space associated with the spacelike soliton of Figure \ref{f3a}. For better visualization, in Figure \ref{f3d} we project the curve of Figure \ref{f3c} onto the canonical plane, which does not contain the hyperboloid's rotation axis. As $k(s)$ has infinitely many zeros, we have infinitely many soliton solutions to the ICF on $\mathbb{H}^{2}$ associated with the spacelike soliton of Figure \ref{f3a} on $\mathbb{S}^{2}_1$.

\begin{figure}[ht]
	\centering	
	\subfloat[]{\includegraphics[height=2.3 cm]{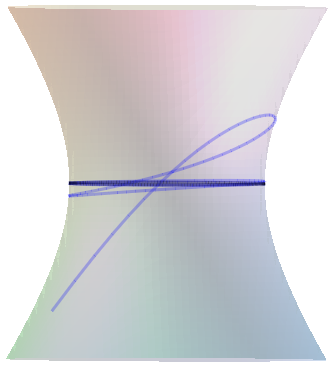}
\label{f3a}
} 
\quad
\subfloat[]{
\includegraphics[height=2.3 cm]{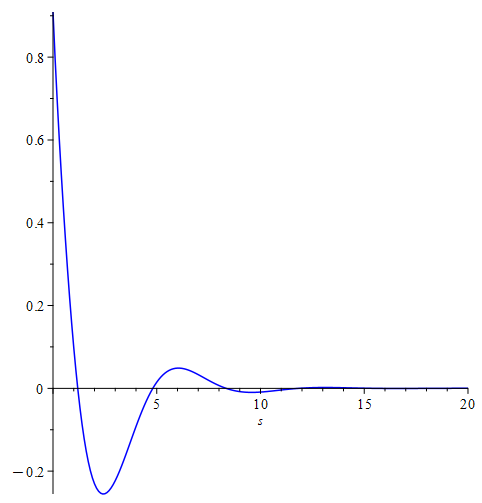}
\label{f3b}
}
\quad	
	\subfloat[]{
		\includegraphics[height=2.1 cm]{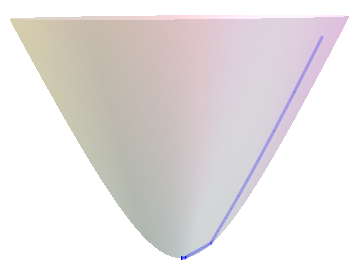}
\label{f3c}
	}
\quad	
	\subfloat[]{
		\includegraphics[height=2.3 cm]{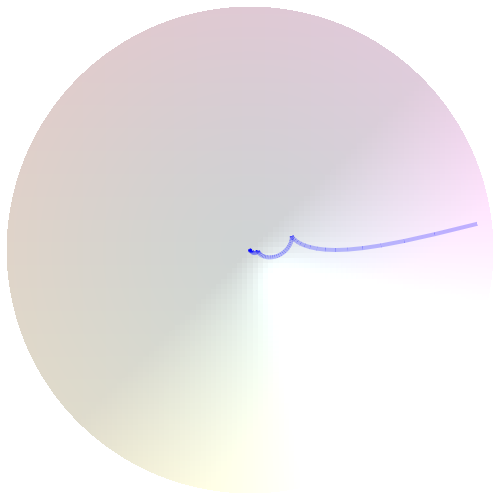}
\label{f3d}
	}
	\caption{Spacelike soliton solution to the CF on $\mathbb{S}^{2}_{1}$ and soliton solution to the ICF on $\mathbb{H}^{2}$ with $a=\frac{10}{11}$, $v=(-1,0,0)$.}
	\label{fig11}
\end{figure} 
In Figure \ref{f4a}, we present a spacelike soliton to the CF such that the curvature has at most finitely many zeros, that is, $v$ is a timelike vector and $a\geq 2$. In Figure \ref{f4b}, we present the graph of the curvature function, and in Figure \ref{f4c}, we present some soliton solutions to the ICF on the 2-dimensional hyperbolic space associated with the spacelike soliton from Figure \ref{f4a}. For better visualization, in Figure \ref{f4c} we project the curve of Figure \ref{f4c} onto the canonical plane, which does not contain the hyperboloid's rotation axis.
\begin{figure}[ht]
	\centering	
	\subfloat[]{\includegraphics[height=2.3 cm]{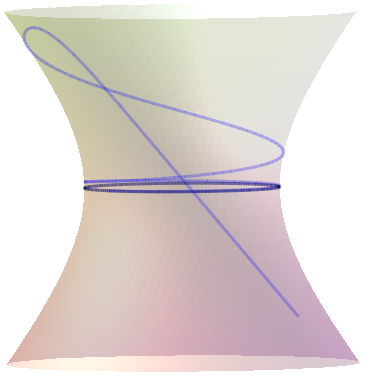}
\label{f4a}
} 
\quad
\subfloat[]{
\includegraphics[height=2.3 cm]{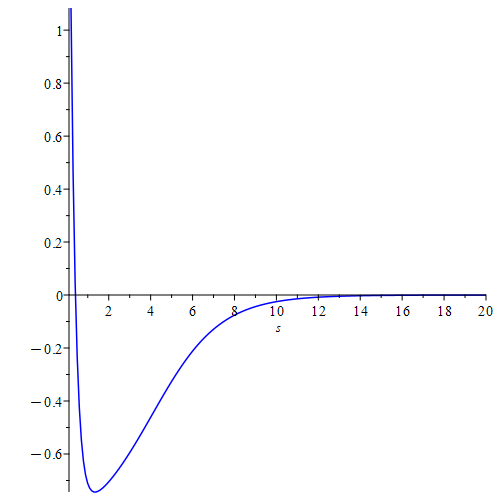}
\label{f4b}
}
\quad	
	\subfloat[]{
		\includegraphics[height=2.5 cm]{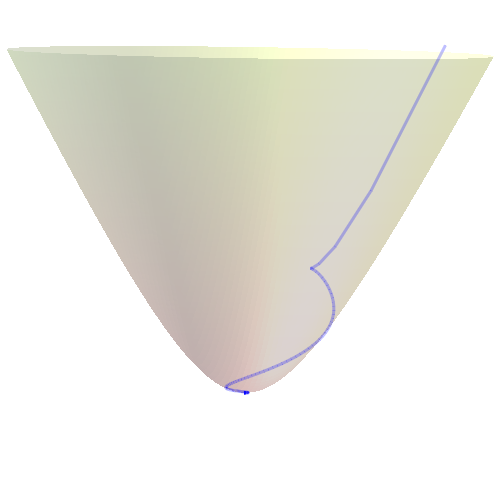}
\label{f4c}
	}
\quad	
	\subfloat[]{
		\includegraphics[height=2.3 cm]{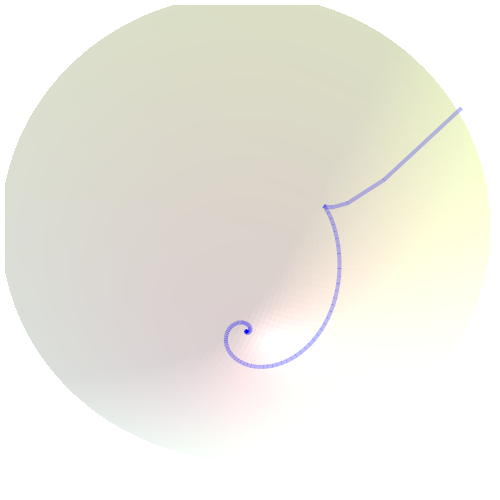}
\label{f4d}
	}
	\caption{Spacelike soliton solution to the CF on $\mathbb{S}^{2}_{1}$ and soliton solution to the ICF on $\mathbb{H}^{2}$ with $a=\frac{7}{3}$, $v=(-1,0,0)$.}
	\label{fig12}
\end{figure} 

In Figure \ref{f5a}, we present a spacelike soliton to the CF, where $v$ is a lightlike vector. In Figure \ref{f5b}, we present the graph of the curvature function, and in Figure \ref{f5c}, we present some soliton solutions to the ICF on the 2-dimensional hyperbolic space associated with the spacelike soliton of Figure \ref{f5a}. For better visualization, in Figure \ref{f5c} we project the curve of Figure \ref{f5c} onto the canonical plane, which does not contain the hyperboloid's rotation axis.

\begin{figure}[ht]
	\centering	
	\subfloat[]{\includegraphics[height=2.3 cm]{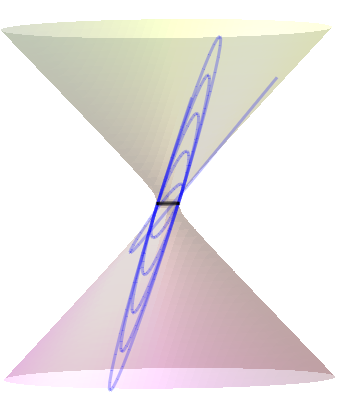}
\label{f5a}
} 
\quad
\subfloat[]{
\includegraphics[height=2.3 cm]{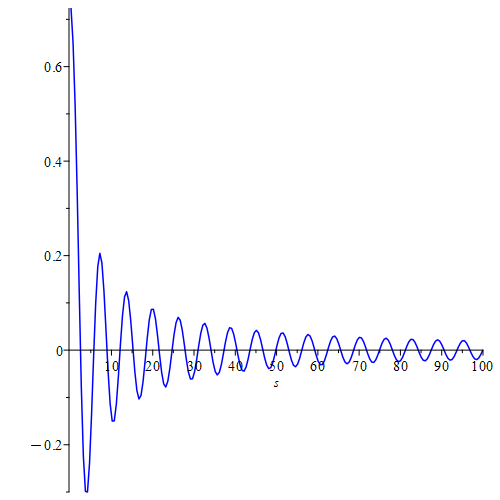}
\label{f5b}
}
\quad	
	\subfloat[]{
		\includegraphics[height=2.3 cm]{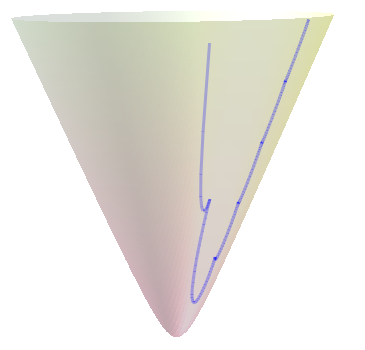}
\label{f5c}
	}
\quad	
	\subfloat[]{
		\includegraphics[height=2.3 cm]{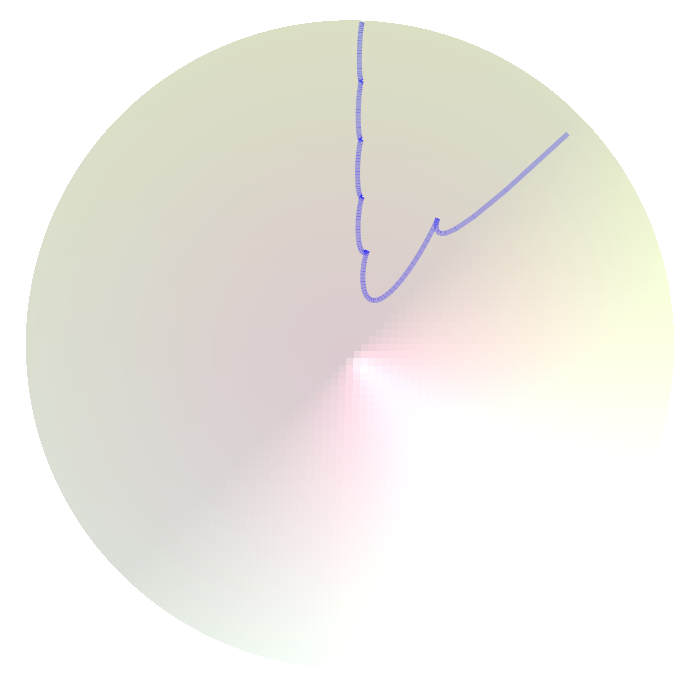}
\label{f5d}
	}
	\caption{Spacelike soliton solution to the CF on $\mathbb{S}^{2}_{1}$ and soliton solution to the ICF on $\mathbb{H}^{2}$ with $a=2$, $v=(-1,1,0)$.}
	\label{fig13}
\end{figure} 

In Figure \ref{f6a}, we present the spacelike soliton to the CF, where $v$ is a spacelike vector and the curvature function has no zeros. In Figure \ref{f6b}, we present the graph of the curvature function, and in Figure \ref{f6c}, the soliton solutions to the ICF on the 2-dimensional hyperbolic space associated with the spacelike soliton of Figure \ref{f6a}.

\begin{figure}[ht]
	\centering	
	\subfloat[]{\includegraphics[height=2.3 cm]{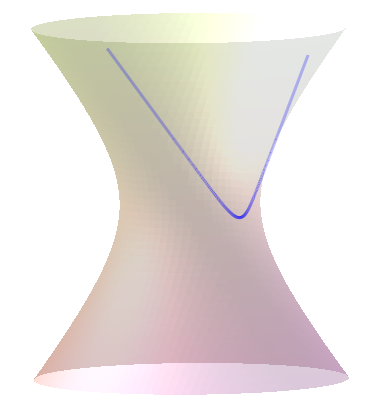}
\label{f6a}
} 
\quad
\subfloat[]{
\includegraphics[height=2.3 cm]{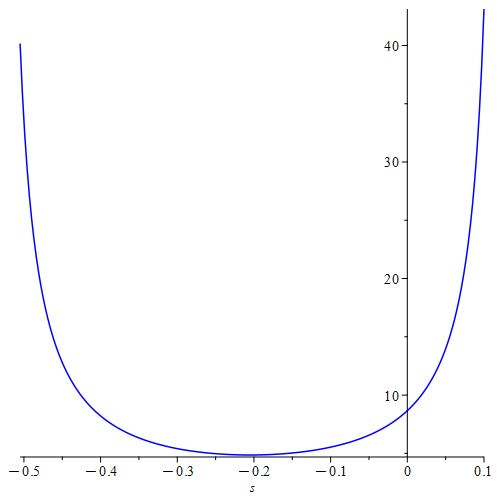}
\label{f6b}
}
\quad	
	\subfloat[]{
		\includegraphics[height=2.3 cm]{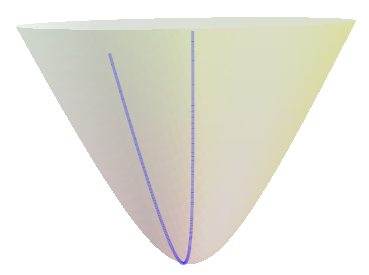}
\label{f6c}
	}
	\caption{Spacelike soliton solution to the CF on $\mathbb{S}^{2}_{1}$ and soliton solution to the ICF on $\mathbb{H}^{2}$ with $a=5$, $v=(0,0,1)$.}
	\label{fig14}
\end{figure}

Finally, in Figure \ref{f7a}, we present the spacelike soliton to the CF, where $v$ is a spacelike vector. In Figure \ref{f7b}, we present the graph of the curvature function. In Figure \ref{f7c}, we show some soliton solutions to the ICF on the 2-dimensional hyperbolic space associated with the spacelike soliton of Figure \ref{f7a}.

\begin{figure}[ht]
	\centering	
	\subfloat[]{\includegraphics[height=2.3 cm]{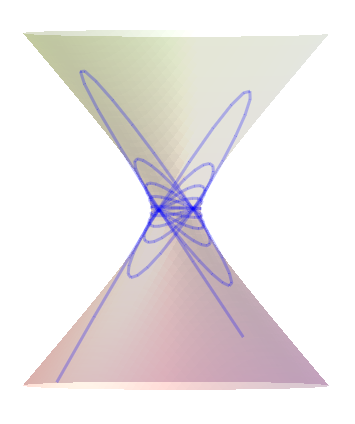}
\label{f7a}
} 
\quad
\subfloat[]{
\includegraphics[height=2.3 cm]{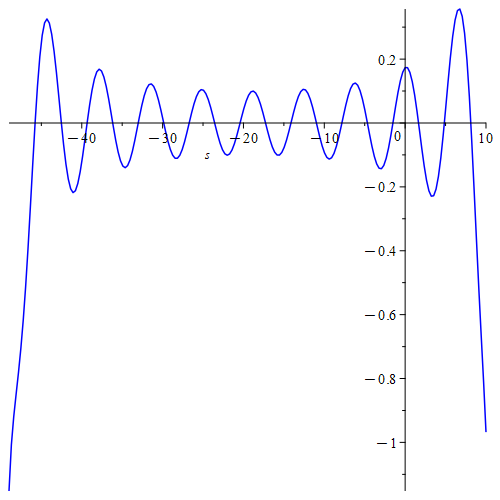}
\label{f7b}
}
\quad	
	\subfloat[]{
		\includegraphics[height=2.3 cm]{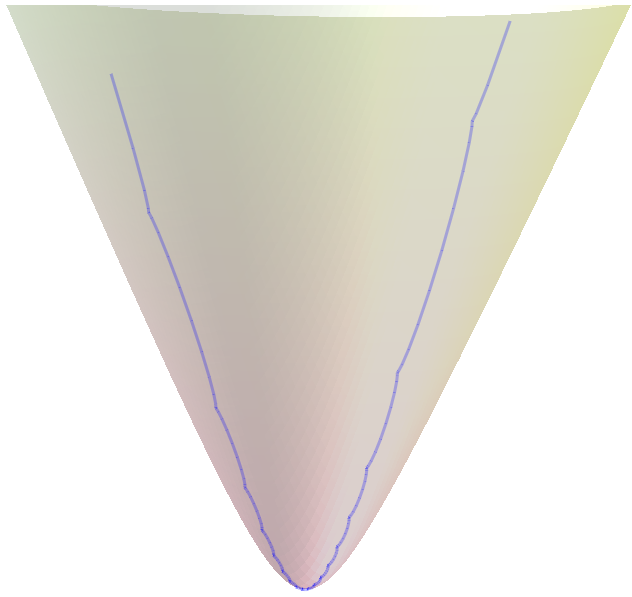}
\label{f7c}
	}
	\caption{Spacelike soliton solution to the CF on $\mathbb{S}^{2}_{1}$ and soliton solution to the ICF on $\mathbb{H}^{2}$ with $a=\frac{1}{10}$, $v=(0,0,1)$.}
	\label{fig15}
\end{figure}

\bibliographystyle{plain}
\bibliography{references} 

\end{document}